\documentclass [11pt,a4paper]{article}
\usepackage[ansinew]{inputenc}
\usepackage{amssymb}
\usepackage{amsmath}
\usepackage{graphicx}
\usepackage{amsthm}
\usepackage{cite}
\newtheorem{theorem}{\textbf{Theorem}}[section]
\newtheorem{lemma}{\textbf{Lemma}}[section]
\newtheorem{proposition}{\textbf{Proposition}}[section]
\newtheorem{corollary}{\textbf{Corollary}}[section]
\newtheorem{remark}{\textbf{Remark}}[section]
\newtheorem{definition}{\textbf{Definition}}[section]

\allowdisplaybreaks[4]

\def\be{\begin{equation}}
\def\ee{\end{equation}}
\def\bea{\begin{eqnarray}}
\def\eea{\end{eqnarray}}
\def\bt{\begin{theorem}}
\def\et{\end{theorem}}
\def\bl{\begin{lemma}}
\def\el{\end{lemma}}
\def\br{\begin{remark}}
\def\er{\end{remark}}
\def\bp{\begin{proposition}}
\def\ep{\end{proposition}}
\def\bc{\begin{corollary}}
\def\ec{\end{corollary}}
\def\bd{\begin{definition}}
\def\ed{\end{definition}}

 \def\non{\nonumber }

\def \au {\rm}

\def \no#1#2#3 {{\bf #1} (#3), #2.}
\def \eds#1#2#3 {#1, #2, #3.}

\begin{document}

\title{Analysis of a diffuse-interface model for the binary viscous incompressible fluids with thermo-induced Marangoni effects}
\author{
{\sc Hao Wu} \footnote{School of Mathematical Sciences and Shanghai
Key Laboratory for Contemporary Applied Mathematics, Fudan
University, 200433 Shanghai, China, Email:
\textit{haowufd@yahoo.com}.}\ \  and {\sc Xiang Xu}
\footnote{Department of Mathematical Sciences, Carnegie Mellon
University, Pittsburgh, PA 15213, Email:
\textit{xuxiang@andrew.cmu.edu}.}}
\date{\today}
\maketitle

\begin{abstract}
In this paper we study the well-posedness and long-time dynamics of
a diffuse-interface model for the mixture of two viscous
incompressible Newtonian fluids with thermo-induced Marangoni
effects. The governing system consists of the Navier--Stokes
equations coupled with phase-field and energy transport equations.
We first derive an energy inequality that illustrates the
dissipative nature of the system under the assumption that the
initial temperature variation is properly small. Then we establish
the existence of weak/strong solutions and
discuss the long-time behavior as well as the stability of the system.

{\bf Keywords.} Phase-field model, Navier--Stokes equations,
Marangoni effects, well-posedness, long-time dynamics.

{\bf Subject Classifications.} 35Q35, 35K55, 76D05.

\end{abstract}

\section{Introduction}
\noindent The study of interface dynamics is of great importance in
the hydrodynamic theory of complex fluids. In the classical
approaches (e.g., the sharp-interface model), the interface is usually considered to be a $n-1$
dimensional free surface of zero width that evolves in time with the
fluid. The resulting hydrodynamic system describing the mixture of
two immiscible Newtonian fluids with a free interface usually
consists of Navier--Stokes equations in each fluid domain with
kinematic and force balance boundary conditions on the interface. On
the other hand, the so-called diffuse-interface model (or
phase-field model), recognizes micro-scale mixing of the
macroscopically immiscible fluids and the interface represents a
thin region with a steep transition property between two fluids (cf.
Anderson et al \cite{AMW}). Within this region, the fluid is mixed
and has to store certain ``mixing energy". The diffuse-interface
model can be viewed as a physically motivated level-set method that
describes the interface by a proper mixing energy. Compared with the
sharp-interface model, the diffuse-interface model can describe
topological transitions of interfaces (like pinchoff and
reconnection) in a natural way (cf. Lowengrub et al \cite{LT98}) and it
has many advantages in numerical simulations of the interfacial
motion (cf. \cite{YFLS04,LS03, FLSY05} and references therein).

The Marangoni effect was initially observed by
Thomson \cite{T1855} during the study of the interesting phenomenon
``tears of wine". Afterwards this phenomenon was defined in more
detail in Marangoni \cite{M1871} in terms of surface tension
gradients and named after the author. The Marangoni effect is a
phenomenon whereby mass transfer occurs due to differences in
surface tension. Such differences can either be attributed to
non-uniform distributions of surfactants (cf. Mendes-Tatsis and Agble
\cite{MA00}) or to the existence of temperature gradient in the
neighborhood of the interface (cf. Sterling and Scriven \cite{SS59}).
The latter is called the thermo-capillary convection or the 
Marangoni--Benard convection, which becomes more and more important
in the application of complex fluids, liquid-gas systems and
ocean-geophysical dynamics (cf. e.g., \cite{BB03, BB07, KT08, JN99, YFLS05}).

The conventional Marangoni--Benard convection for the mixture of two
Newtonian flows can be described by a sharp-interface model
involving the Boussinesq approximation (cf. e.g., Liu et al \cite{LSFY05})
\bea &&\rho(u_t+(u\cdot\nabla) u)+\nabla p-\nu \Delta u =-\rho_\theta g\mathbf{j}, \label{sharp1} \\
&&\nabla\cdot u=0, \label{sharp2} \\
&&\theta_t+u\cdot\nabla\theta=k\Delta\theta, \label{sharp3}
 \eea
  where $u$, $p$ and $\theta$ stand for the fluid velocity, the pressure, and the relative temperature (with respect to the reference background temperature $\theta_b$, which is assumed
to be a constant for the sake of simplicity), respectively. $\rho$
is the density of fluid mixture, $\nu$ is the viscosity, $g$ is the
gravitational acceleration, $\mathbf{j}$ is the upward direction and
$k>0$ is the thermal diffusion constant. We assume that the temperature-dependent
density $\rho_\theta$ is described by the Boussinesq approximation
 \be
 \rho_\theta=\rho(1-\alpha\theta),
 \ee
where $\alpha$ is the coefficient of thermal expansion. The
background density $\rho$ is assumed to be a constant and the
difference between the actual density and $\rho$ only contributes to
the buoyancy force. Interface conditions are given by
 \bea
 &&  l_t+u\cdot \nabla l=0, \label{sharp4} \\
 &&  [T]\cdot\mathbf{n}=-\sigma H \mathbf{n}+ (\tau\cdot\nabla \sigma)\tau,\label{sharp5}
 \eea
 where $l$ stands for the interface length of the mixture. The kinematic condition \eqref{sharp4} indicates that  the surface ($l=0$) evolves with the fluid. Equation \eqref{sharp5} is the balance
of forces on the interface, where $H$ is the curvature of the
interface, $[T]$ is the jump of the stress across the interface, $\sigma$ is the surface tension, $\tau$ is the tangential direction on the interface and $\mathbf{n}$ is
the normal direction.

In this paper, we shall investigate a diffuse-interface model, which was used to describe the thermo-induced Marangoni effects in the mixture of two incompressible Newtonian
fluids. A phase-field variable $\phi$ is introduced as the
volume fraction to demarcate the two species and to indicate the
location of the interface. The region $\{x : \phi(x,t)=1\}$ is occupied by fluid 1 while $\{x : \phi(x,t)=-1\}$ is occupied by fluid 2. The
interface is represented by $\{x : \phi(x,t)=0\}$, with a (fixed)
transition layer of thickness $\varepsilon$. In the
diffuse-interface approach, one usually introduces an elastic (mixing)
energy of Ginzburg--Landau type
 \be
 E(\phi)=\int_\Omega
\left[\frac12|\nabla \phi|^2+F(\phi)\right]dx,\label{elastic}
 \ee
which represents the competition between the hydrophobic and
hydrophilic effects of the two different species. The physically
relevant energy density function $F$ that represents the two phases of the mixture usually has a double-well structure. A typical
example of $F$ is the so-called logarithmic potential (cf.
Cahn and Hillard \cite{58})
 \be F(\phi)=\gamma_1(1-\phi^2)+\gamma_2[(1+\phi)\ln (1+\phi)+(1-\phi)\ln
(1-\phi)],\quad \gamma_1,\gamma_2>0.\non
  \ee
  In practice, this (singular) potential is often replaced by a
smooth double-well polynomial approximation
 \[ F(\phi)=\frac{1}{4\varepsilon^2} (\phi^2-1)^2. \]
In this paper, we start with the simple case that the two components of the binary fluid have matched densities and the same constant viscosity $\nu$ as well as the same constant heat conductivity $k$. As a consequence, we consider the following system (cf. e.g., \cite{SLX09,LSFY05,
LS03}):
\bea &&\rho(u_t+(u\cdot\nabla) u)+\nabla p-\nu\Delta{u} \non\\
&&\ \
=-\nabla\cdot\left[\lambda(\theta)\nabla\phi\otimes\nabla\phi-\lambda(\theta)\Big(\frac12|\nabla\phi|^2+
F(\phi)\Big)\mathbb{I}
\right]-\rho_\theta g\mathbf{j}, \label{navier-stokes} \\
&&\nabla\cdot u=0, \label{incompressible} \\
&&\phi_t+u\cdot\nabla\phi=\gamma\big(\Delta\phi-F'(\phi)\big),
\label{phase}
\\
&&\theta_t+u\cdot\nabla\theta=k\Delta\theta, \label{temperature}
 \eea
  for $(x,t)\in \Omega\times (0,+\infty)$.  Here, we assume that $\Omega$ is a bounded domain in $\mathbb{R}^n$ ($n=2,
3$) with smooth boundary $\Gamma$. The usual Kronecker product is
denoted by $\otimes$, i.e., $(a\otimes b)_{ij}=a_ib_j$ for $a,b \in
\mathbb{R}^n$. The system \eqref{navier-stokes}--\eqref{temperature}
contains the Navier--Stokes equations, an Allen--Cahn type equation
for the phase-field function and an energy transport equation for
the temperature. The parameter $\gamma$ represents the microscopic
elastic relaxation time due to the presence of the microstructure of
the mixture. As $\gamma\to 0$, the internal dissipative mechanism
will disappear and the limiting equation is equivalent to the mass
transport equation for incompressible fluids (cf. \cite{LSFY05, YFLS05}). The temperature-dependent surface tension
coefficient $\lambda(\theta)$ is supposed to be
\be
 \lambda(\theta)=\lambda_0(a-b\theta),
 \non
 \ee
 where $\lambda_0>0$, $a>0$, $b \neq 0$ are constants. Usually $\lambda_0$ is assumed to be proportional
 to the interface length $\varepsilon$ (cf. Sun et al \cite{SLX09}). Numerical experiments have been made in the recent paper Sun et al
\cite{SLX09} to illustrate the role played by thermal energy in the
interfacial dynamics of two-phase flows due to the thermo-induced
surface tension heterogeneity on the interface. Their results suggest that the system
\eqref{navier-stokes}--\eqref{temperature} (and its generalizations)
turns out to be a suitable mathematical representation that reflects the thermo-induced
Marangoni effects in the mixture of fluids.

  We suppose that the system \eqref{navier-stokes}--\eqref{temperature} is subject to the initial conditions
 \be  u|_{t=0}=u_0(x) \mbox{ with } \nabla\cdot u_0=0, \ \ \ \phi|_{t=0}=\phi_0(x), \ \ \ 
 \theta|_{t=0}=\theta_0(x), \quad \ x\in \Omega. \label{IC}
  \ee
 Moreover, we assume no-slip boundary condition
on the velocity $u$
  \be
  u(x, t)=0, \quad  (x, t) \in \Gamma \times (0, +\infty), \label{BC1}
  \ee
  nonhomogeneous Dirichlet boundary condition on
 the phase function $\phi$
 \be
   \phi(x, t)=-1,  \quad  (x, t) \in \Gamma \times (0, +\infty), \label{BC2}
  \ee
 and homogeneous Dirichlet boundary condition on
 the temperature $\theta$
 \be
   \theta(x, t)=0,  \quad  (x, t) \in \Gamma \times (0, +\infty).
   \label{BC}
 \ee

  The goal of this paper is to provide a detailed mathematical
theory of existence, uniqueness, regularity and long-time behavior
of solutions to the non-isothermal
Navier--Stokes--Allen--Cahn system
\eqref{navier-stokes}--\eqref{BC}. First, we prove the
existence of global weak solutions in two and three spatial
dimensions (cf. Theorem \ref{theorem on
weak existence}). Next, we obtain the existence and uniqueness of a global strong solution in $2D$ (cf. Theorem \ref{strong2D}), a local strong solution in $3D$ (cf. Theorem \ref{locstrong}) and a global one provided that the viscosity $\nu$ is properly large (cf. Theorem \ref{theorem on large viscosity case}). The long-time dynamics of the system seems to be more complicated
than the uncoupled Navier--Stokes equations. We prove that as $t\to+\infty$, the phase function $\phi$ converges to
a solution of the stationary Allen--Cahn equation and the velocity $u$ as well as the temperature $\theta$ converges to zero (cf. Theorem \ref{theorem on long time behavior}). Stability for minimizers of the elastic energy is also discussed (cf. Theorem \ref{theorem on small data}). We just remark that our results can be easily extended to the cases with more general Dirichlet boundary conditions for the phase function. For instance, \eqref{BC2} can be replaced by $ \phi(x, t)|_\Gamma =h(x)$ on $\Gamma \times (0, +\infty)$, with $h(x)=\phi_0(x)|_\Gamma$ and $h(x)\in H^\frac32(\Gamma)$, $|h(x)|\leq 1$ (we refer to Lin and Liu \cite{LL95} for a similar situation for a simplified nematic liquid crystal system).

It is easy to
verify that for the isothermal case of system
\eqref{navier-stokes}--\eqref{BC} without the Boussinesq
approximation term, there is a dissipative energy law
 \bea
  && \frac{d}{dt}\left(\frac12\|u\|^2+
  \frac{\lambda}{2}\|\nabla \phi\|^2+ \lambda \int_\Omega F(\phi) dx\right)\non\\
  &=& -\nu\int_\Omega |\nabla{u}|^2dx-\lambda\gamma\int_\Omega |-\Delta \phi +F'(\phi)|^2 dx.\label{bbel}
 \eea
  This basic energy law reveals the underlying physics for the isothermal Navier--Stokes--Allen--Cahn system and it plays an important role in the study of well-posedness as well as long-time behavior of the system. We refer to the recent work \cite{GG10, GG10b} for detailed mathematical analysis on an isothermal NSAC system for the incompressible two-phase flows, where long-time
behavior of global solutions was analyzed within the theory of infinite-dimensional
dissipative dynamical systems (e.g., the existence of
global attractors, exponential attractors, trajectory attractors and
convergence to single equiblibria).

However, in our present
case, the surface tension parameter $\lambda$ in \eqref{navier-stokes} depends on the temperature such that it is no longer a constant. Moreover, the Boussinesq approximation is also applied. These bring us challenges in mathematical analysis of the system. We are not able to derive the same dissipative energy equality as for the isothermal case. In
particular, the special cancelation between the induced stress term
in the momentum equation and the transport term in the phase-field equation is no longer valid (see Remark \ref{can} below). This relation is crucial to derive the 
dissipative energy law like \eqref{bbel} (cf. \cite{GG10, LL95,
YFLS04}).
 Nevertheless, taking advantage of proper maximum principles for the phase function and temperature, we show that
 if the initial temperature variation is not large (bounded in terms of coefficients of the system), we can derive an energy \emph{inequality} for the system
\eqref{navier-stokes}--\eqref{BC}, which reflects the dissipative
nature of the flow.

We remark that in the phase-field equation \eqref{phase}, the dynamics of the
phase function $\phi$ is assumed to be driven by a gradient flow of
Allen--Cahn type. In order to keep the conservation of overall volume
fraction, people usually assume that the internal dissipation is
described through a Cahn--Hillard equation (with convection) for $\phi$, which can be viewed as a gradient flow of the elastic
energy in the Sobolev space $H^{-1}$ (cf. \cite{LS03,LT98}). The resulting system is
nevertheless much more involved in mathematical analysis, because it contains a
fourth-order differential operator and thus the maximum principle for the phase function $\phi$ no longer holds (see, for instance, \cite{Ab,B,GG10a} for the
isothermal Navier--Stokes--Cahn--Hilliard system). It seems that our results cannot be extended to this case in a straightforward way. For instance, in the derivation of the dissipative energy inequality \eqref{basic energy inequality}, we rely on the $L^\infty$-estimate of $\phi$, which will no longer be available due to the lack of maximum principle. This problem will be studied in our future work.

The rest of the paper is organized as follows. In Section 2, we derive
an energy inequality that guarantees the dissipative nature of the
system and establish the existence of global weak solutions. In
Section 3, we discuss existence and uniqueness of strong solutions
in both $2D$ and $3D$. In Section 4, we study the long-time dynamics
and stability property of the system.

\section{Global weak solutions}\setcounter{equation}{0}
If $X$ is a real Hilbert space with inner product $(\cdot,\cdot)_X$,
then we denote the induced norm by $\|\cdot\|_X$. $X'$ indicates the
dual space of $X$ and $\langle\cdot, \cdot\rangle_{X',X}$ will denote the
corresponding dual product. We indicate by $\mathbf{X}$ the
vectorial space $X^n$ endowed with the product structure. For
simplicity, the scalar product in $L^2(\Omega)$ (also
$\mathbf{L}^2(\Omega)$) will be denoted by  $(\cdot,\cdot)$, and the
associated norm
 by $\|\cdot\|$. For two $n\times n$ matrices $M_1, M_2$, we denote $M_1 : M_2={\rm trace}(M_1 M_2^T)$.
Let $\mathcal{V}=C_0^\infty(\Omega, \mathbb{R}^n) \cap \{v :
\nabla\cdot v=0\}$. We denote $\mathbf{H}$ (respectively
$\mathbf{V}$) the closure of $\mathcal{V}$ in $\mathbf{L}^2$
(respectively $\mathbf{H}^1$).
 \be
\mathbf{H}=\{u\in \mathbf{L}^2: \nabla \cdot u=0, \ \ u\cdot
\mathbf{n}=0 \ \text{on}\ \Gamma\},\quad \mathbf{V}=\{u\in
\mathbf{H}_0^1: \nabla \cdot u=0\}.\non
 \ee
 $\mathbf{H}$ and $\mathbf{V}$ are Hilbert spaces with norms $\|\cdot\|$ and $\|\cdot\|_{\mathbf{H}^1}$, respectively.
We recall the Stokes operator $S: \mathbf{H}^2(\Omega) \cap
\mathbf{V} \rightarrow \mathbf{H}$ such that $Su=-\Delta u+\nabla\pi
\in \mathbf{H}$, for all $u \in \mathbf{H}^2(\Omega) \cap
\mathbf{V}$. $S^{-1}$ is a compact linear operator on $\mathbf{H}$
and $\|S\cdot\|$ is a norm on $D(S)$ that is equivalent to the
$\mathbf{H}^2$-norm. Then there exists a positive constant $C=C(n, \Omega)$, for which
(cf. Temam \cite{Te01})
 \be \|u\|_{\mathbf{H}^2}+\|\pi\|_{H^1
\backslash {\mathbb{R}}} \leq C\|Su\|. \label{Stokes II}
 \ee
In the following text, we denote by $C$, $C_i$ the generic constants
depending on $a$, $b$, $g$, $k$, $\lambda_0$, $\alpha$, $\gamma$,
$\Omega$, $\varepsilon$ and the initial data. Special dependence
will be pointed out explicitly in the text if necessary.

  Without loss of generality, we assume $\rho=1$ in the remaining part of this paper. Now we introduce the weak formulation of the initial boundary value problem
\eqref{navier-stokes}--\eqref{BC}:

\begin{definition}
For any $T \in (0, +\infty)$, the triple $(u, \phi, \theta)$
satisfying \bea &&u \in L^\infty(0, T; \mathbf{H}) \cap L^2(0, T;
\mathbf{V}), \non \\
&&\phi\in L^\infty(0, T; {H}^1\cap L^\infty) \cap L^2(0,
T; {H}^2),\quad |\phi|\leq 1, \ \text{a.e. in} \ \Omega\times[0,T],\non\\
&& \theta\in L^\infty(0, T; {H}_0^1\cap L^\infty) \cap L^2(0, T;
{H}^2),\non  \eea is called a weak solution of the problem
\eqref{navier-stokes}--\eqref{temperature} if
 the initial and boundary conditions \eqref{IC}--\eqref{BC} are satisfied and for a.e. $t\in (0, T)$,
 \bea
 && \langle u_t, v\rangle_{\mathbf{V}', \mathbf{V}}+ \int_\Omega (u\cdot\nabla)u\cdot v dx +\nu\int_\Omega \nabla{u} : \nabla{v} dx \non\\
 && \quad = \int_\Omega [\lambda(\theta)\nabla\phi\otimes\nabla\phi] : \nabla v dx+\alpha g\int_\Omega  \theta \mathbf{j} \cdot v dx,\quad \forall\,
 v\in \mathbf{V},\non\\
 && \phi_t+ u\cdot \nabla\phi=\gamma (\Delta \phi- F'(\phi)), \quad \text{a.e. in\ } \Omega, \non\\
 &&  \theta_t+ u\cdot \nabla\theta=k\Delta\theta,\quad \text{a.e. in\ } \Omega,\non\\
 && u|_{t=0}=u_0(x),\ \ \phi|_{t=0}=\phi_0(x), \ \ \theta|_{t=0}=\theta_0(x), \quad \text{in}\ \Omega.\non
 \eea
 \end{definition}
 \begin{remark}
  In order to derive the variational formulation for $u$,
  we use the following facts due to the incompressibility condition: for any $ v\in \mathbf{V}$, it holds
 \bea
 &&\int_\Omega \nabla p \cdot vdx= \int_\Omega \mathbf{j}\cdot v dx= 0,\non\\
 && \int_\Omega \left[\nabla \cdot \left(\lambda(\theta)\Big(\frac12|\nabla\phi|^2+ F(\phi)\Big)\mathbb{I}\right)\right]\cdot v\,dx
 = \int_\Omega \nabla \left[\lambda(\theta)\Big(\frac12|\nabla\phi|^2+ F(\phi)\Big)\right]\cdot v \, dx=0.\non
\eea
 \end{remark}

Next, we state the result on the existence of global-in-time weak
solutions:

\begin{theorem}[Existence of weak solutions] \label{theorem on weak existence}
Suppose $n=2, 3$. For any initial data $(u_0, \phi_0, \theta_0)$
$\in$ $\mathbf{H} \times (H^1(\Omega)\cap L^\infty(\Omega))\times
(H^1_0(\Omega)\cap L^\infty(\Omega))$ satisfying
 \be \|\phi_0\|_{L^\infty}\leq 1,\quad \|\theta_0\|_{L^\infty}\leq   \frac{1}{4C_1^2|b|}\sqrt{\frac{a\gamma\nu}{2\lambda_0}},\label{ini}
  \ee
where $C_1$ is a constant depending only on $n, \Omega$, the problem
\eqref{navier-stokes}--\eqref{BC} admits at least one global weak
solution such that \bea &&u \in L^\infty(0, +\infty; \mathbf{H})
\cap L^2_{loc}(0, +\infty;
\mathbf{V}), \non \\
&&\phi\in L^\infty(0, +\infty; {H}^1\cap L^\infty) \cap L^2_{loc}(0,
+\infty; {H}^2),\quad |\phi|\leq 1, \ \text{a.e. in} \ \Omega\times[0,+\infty),\non\\
&& \theta\in L^\infty(0, +\infty; {H}_0^1\cap L^\infty) \cap
L^2_{loc}(0, +\infty; {H}^2).\non \eea
\end{theorem}

\subsection{Dissipative energy inequality}
An important feature of problem \eqref{navier-stokes}--\eqref{BC} is that $\phi$ and $\theta$
satisfy the following weak maximum principles, which will be
useful in the derivation of the dissipative energy inequality for the system.
\begin{lemma} \label{mphi}
Suppose $u \in L^\infty(0, T; \mathbf{H}) \cap L^2(0, T;
\mathbf{V})$. If $\phi\in L^\infty(0,T; H^1\cap L^\infty)\cap L^2(0,
T; H^2)$ is the weak solution of the initial boundary value problem
 \bea
&& \phi_t+ u\cdot \nabla\phi=\gamma (\Delta \phi- F'(\phi)), \quad \text{a.e. in\ } \Omega,\non\\
&& \phi(x,t)|_\Gamma=-1, \quad (x, t)\in \Gamma\times (0,T), \non\\
&& \phi|_{t=0}=\phi_0(x)\in  H^1\cap L^\infty, \ \text{with}\ \
|\phi_0|\leq 1\ \text{a.e. in} \ \Omega,\non
 \eea
then  $|\phi(x,t)| \leq 1$, a.e. in $\Omega$ for each $t \in (0,
T)$.
\end{lemma}
\begin{proof} The proof is similar to that for the liquid crystal system
(cf. e.g., \cite{LL95,C09}), so we omit the details here.
\end{proof}
\begin{lemma} \label{mtheta}
Suppose $u \in L^\infty(0, T; \mathbf{H}) \cap L^2(0, T;
\mathbf{V})$. If $\theta\in L^\infty(0,T; H^1_0\cap L^\infty)\cap
L^2(0, T; H^2)$ is the weak solution of  the initial boundary value
problem
 \bea
&&\theta_t+ u\cdot \nabla\theta=k\Delta \theta,\quad \text{a.e. in\ } \Omega,\non\\
&&\theta|_\Gamma=0, \quad (x, t)\in \Gamma\times (0,T), \non\\
&& \theta|_{t=0}=\theta_0(x)\in H^1_0\cap L^\infty,\non
 \eea
 then $\|\theta(t)\|_{L^\infty}\leq \|\theta_0\|_{L^\infty}$ for every $t\in (0,T)$.
\end{lemma}
\begin{proof}
 Multiplying the equation by $|\theta|^{q-1}\theta$ ($q>1$), integrating over $\Omega$, we get
 \be
 \frac{1}{1+q}\frac{d}{dt}\int_\Omega |\theta|^{1+q} dx + \frac{1}{1+q}\int_\Omega u\cdot \nabla |\theta|^{q+1} dx +\frac{4(q-1)}{(1+q)^2}\int_\Omega k\left|\nabla(|\theta|^\frac{q-1}{2}\theta)\right|^2dx=0,\non
 \ee
 which implies that
 \be
 \|\theta(t)\|_{L^q}\leq \|\theta_0\|_{L^q}\leq |\Omega|^\frac{1}{1+q}\|\theta_0\|_{L^\infty}, \quad \forall\, q>1, \ \ t\in (0,T).\non
 \ee
 Taking the limit $q\to+\infty$, we arrive our conclusion.
 \end{proof}
In what follows, we derive a dissipative energy inequality, which
turns out to be crucial in the study of well-posedness as well as long-time
dynamics of the problem \eqref{navier-stokes}--\eqref{BC}.

 \begin{proposition}[Dissipative energy inequality] \label{BEL} For $n=2, 3$, we assume that the initial phase function $\phi_0$ and the initial temperature
$\theta_0$ satisfy the assumption \eqref{ini}. Then there exist
constants $\zeta, \omega>0$ that depend only on $\Omega$ and
coefficients of the system such that if $(u, \phi, \theta)$ is a
smooth solution to the problem \eqref{navier-stokes}--\eqref{BC}, then
the following energy inequality holds:
 \be
\frac{d \mathcal{E}}{dt} \leq -\frac{\nu}{2}\|\nabla
u\|^2-a\lambda_0\gamma\|\Delta\phi-F'(\phi)\|^2-k\zeta\|\Delta\theta\|^2\leq
0, \quad \forall\, t>0,
 \label{basic energy inequality}
  \ee
  where
  \be \mathcal{E}(t)=
\|u(t)\|^2+a\lambda_0\|\nabla\phi(t)\|^2+2a\lambda_0\int_{\Omega}F(\phi(t))dx
+\zeta\|\nabla\theta(t)\|^2+\omega\|\theta(t)\|^2\geq 0.
 \label{def of total energy}
 \ee
\end{proposition}
\begin{proof} Multiplying \eqref{navier-stokes} with $u$, \eqref{phase} with
$-a\lambda_0(\Delta\phi-F'(\phi))$, \eqref{temperature} with
$-\zeta\Delta\theta$ ($\zeta>0$ is a constant to be determined
later), respectively, adding them up and integrating over $\Omega$,
we have
  \bea
&&\frac12\frac{d}{dt}\Big(\|u\|^2+a\lambda_0\|\nabla\phi\|^2+2a\lambda_0\int_{\Omega}F(\phi)dx+\zeta\|\nabla\theta\|^2
 \Big)\non\\
 &&\ +\nu\|\nabla{u}\|^2+a\gamma\lambda_0\|\Delta\phi-F'(\phi)\|^2+k\zeta\|\Delta\theta\|^2 \non\\
&=&\int_{\Omega}\big(\lambda(\theta)\nabla\phi\otimes\nabla\phi\big):
\nabla u\,dx+\alpha\int_{\Omega}\theta\, g\mathbf{j} \cdot
u\,dx+\zeta\int_{\Omega}(u\cdot\nabla)\theta\Delta\theta \,dx \non\\
&& + a\lambda_0\int_{\Omega}(u\cdot\nabla)\phi(\Delta\phi-F'(\phi)) \,dx\non\\
&:=& J_1+J_2+J_3+J_4.\non
 \eea
 In the following we just treat the case $n=3$, while the
case $n=2$ is similar. Recall the Gagliardo--Nirenberg
inequality
$$ \|\nabla \phi\|_{\mathbf{L}^4}\leq C_1(\|\Delta \phi\|^\frac12\|\phi\|_{L^\infty}^\frac12+\|\phi\|_{L^\infty}), \quad \forall\, \phi\in H^2,$$
where $C_1=C_1(n, \Omega)$ depends only on $n$ and $\Omega$.
Combining it with the Poincar\'e inequality and the Young inequality, we
deduce that
\bea && J_1+J_4\non\\
&=&
-a\lambda_0\int_{\Omega}u\cdot\nabla\left(\frac{|\nabla\phi|^2}{2}+F(\phi)\right)dx
-b\lambda_0\int_{\Omega} \theta \nabla u:(\nabla\phi\otimes \nabla\phi)\,dx\non\\
&\leq& |b|\lambda_0\|\theta\|_{L^\infty}\|\nabla
u\|\|\nabla\phi\|_{\mathbf{L}^4}^2  \non\\
&\leq& \frac{\nu}{4}\|\nabla
u\|^2+\frac{|b|^2\lambda_0^2}{\nu}\|\theta\|_{L^\infty}^2\|\nabla\phi\|_{\mathbf{L}^4}^4
\non\\
&\leq& \frac{\nu}{4}\|\nabla
u\|^2+\frac{8C_1^4|b|^2\lambda_0^2}{\nu}\|\theta\|_{L^\infty}^2\|\Delta\phi\|^2\|\phi\|_{L^\infty}^2
+\frac{8C_1^4|b|^2\lambda_0^2}{\nu}\|\theta\|_{L^\infty}^2 \non\\
&\leq& \frac{\nu}{4}\|\nabla
u\|^2+\frac{16C_1^4|b|^2\lambda_0^2}{\nu}\|\theta\|_{L^\infty}^2(\|\Delta\phi-F'(\phi)\|^2+\|F'(\phi)\|^2)\|\phi\|_{L^\infty}^2
\non\\
&& +\frac{8C_1^4|b|^2\lambda_0^2}{\nu}\|\theta\|_{L^\infty}^2.\non \eea
 Then by Lemmas \ref{mphi}, \ref{mtheta} and the assumption \eqref{ini}, we obtain that
 \bea
 &&
 \frac{16C_1^4|b|^2\lambda_0^2}{\nu}\|\theta\|_{L^\infty}^2\|\Delta\phi-F'(\phi)\|^2\|\phi\|_{L^\infty}^2
 \leq \frac{a\lambda_0\gamma}{2}\|\Delta\phi-F'(\phi)\|^2,\non
\eea
 \bea
 \frac{16C_1^4|b|^2\lambda_0^2}{\nu}\|\theta\|_{L^\infty}^2\|F'(\phi)\|^2\|\phi\|_{L^\infty}^2
 &\leq& \frac{16C_1^4C_2|b|^2\lambda_0^2|\Omega|^2}{\nu\varepsilon^4}\|\Delta \theta\|^\frac32\|\theta\|^\frac12\non\\
 &\leq& \frac{k\zeta}{8}\|\Delta \theta\|^2
 +\frac{54\cdot16^4C_1^{16}C_2^4|b|^8\lambda_0^8|\Omega|^8}{\nu^{
 4}\varepsilon^{16}k^3\zeta^3}\|\theta\|^2,\non \eea
and \bea
\frac{8C_1^4|b|^2\lambda_0^2}{\nu}\|\theta\|_{L^\infty}^2
  &\leq&
 \frac{8C_1^4C_2|b|^2\lambda_0^2}{\nu}\|\Delta\theta\|^{\frac32}\|\theta\|^{\frac12}
 \non\\
 &\leq&\frac{k\zeta}{8}\|\Delta\theta\|^2+\frac{54\cdot 8^4C_1^{16}C_2^4|b|^8\lambda_0^8}{\nu^{
 4}k^3\zeta^3}\|\theta\|^2,\non
 \eea
where $C_2$ depends only on $\Omega$. As a result,
 \be
 J_1+J_4\leq \frac{\nu}{4}\|\nabla
u\|^2+\frac{a\lambda_0\gamma
}{2}\|\Delta\phi-F'(\phi)\|^2+\frac{k\zeta}{4}\|\Delta \theta\|^2+
\frac{C_3}{\zeta^3}\|\theta\|^2,\non
 \ee
 with
 \be
 C_3=\frac{54\cdot 16^4C_1^{16}C_2^4|b|^8\lambda_0^8|\Omega|^8+54\cdot 8^4\varepsilon^{16}C_1^{16}C_2^4|b|^8\lambda_0^8}{\nu^{
 4}\varepsilon^{16}k^3}.\non
 \ee
Next, by the Poincar\'e inequality
 \be
J_2\leq |\alpha||g|\|\theta\|\|u\| \leq
C_P|\alpha||g|\|\theta\|\|\nabla u\|  \leq \frac{\nu}{4}\|\nabla
u\|^2+\frac{C_P^2|\alpha|^2|g|^2}{\nu}\|\theta\|^2,\non
 \ee
where $C_P$ depends only on $\Omega$. For $J_3$, we have
 \bea
J_3&=&-\zeta\int_{\Omega}u
\cdot\nabla\left(\frac{|\nabla\theta|^2}{2}\right)dx
+\zeta\int_{\Omega} u\cdot [\nabla \cdot(\nabla\theta\otimes \nabla\theta)]\,dx\non\\
&=& -\zeta\int_{\Omega} \nabla u: (\nabla\theta\otimes \nabla\theta)\,dx\non\\
&\leq& \zeta\|\nabla u\|\|\nabla\theta\|_{\mathbf{L}^4}^2 \leq
C_1^2\zeta\|\nabla u\|\|\Delta\theta\|\|\theta\|_{L^\infty}
\non\\
&\leq&
\frac{k\zeta}{4}\|\Delta\theta\|^2+\frac{C_1^4\zeta}{k}\|\theta_0\|_{L^\infty}^2\|\nabla
u\|^2\non\\
&\leq&\frac{k\zeta}{4}\|\Delta\theta\|^2+\frac{a\gamma\nu\zeta}{4k|b|^2\lambda_0}\|\nabla
u\|^2.\non
 \eea
 Taking $$\zeta=\frac{k|b|^2\lambda_0}{a\gamma},$$
 we infer from the above estimates that \bea
&&\frac{d}{dt}\left(\|u\|^2+a\lambda_0\|\nabla\phi\|^2+2a\lambda_0\int_{\Omega}F(\phi)dx+\zeta\|\nabla\theta\|^2
 \right)\non\\
 &&\ \ +\frac{\nu}{2}\|\nabla u\|^2
 +\gamma a\lambda_0\|\Delta\phi-F'(\phi)\|^2+k\zeta\|\Delta\theta\|^2 \non\\
 & \leq& C_4\|\theta\|^2,\label{part1 of basic energy law}
\eea where
 $$ C_4=\frac{2C_3}{\zeta^3}+\frac{2C_P^2|\alpha|^2|g|^2}{\nu}.$$
  Multiplying
\eqref{temperature} by $2\omega\theta$,
$\omega=\frac{C_P^2C_4}{2k}>0$, integrating over $\Omega$, and using
the Poincar\'e inequality, we obtain
 \be
\omega\frac{d}{dt}\|\theta\|^2 = -2\omega k\|\nabla\theta\|^2 \leq
-\frac{2\omega k\|\theta\|^2}{C_P^2} = -C_4\|\theta\|^2.
 \label{part2 of basic energy law}
 \ee
Adding \eqref{part1 of basic energy law} with \eqref{part2 of basic
energy law}, we arrive at our conclusion.
 \end{proof}

 \br\label{can}
For the isothermal case of the system
\eqref{navier-stokes}--\eqref{BC} without the Boussinesq
approximation term, there is a special cancelation between the
induced stress term in the Navier--Stokes equations and the
convection term in the Allen--Cahn equation, which yields the
dissipative energy {\rm equality \eqref{bbel}}. However, for the
current non-isothermal system \eqref{navier-stokes}--\eqref{BC},
there exists an extra high-order term
$$-b\lambda_0\int_{\Omega} \theta \nabla u:(\nabla\phi\otimes
\nabla\phi)\,dx$$
containing velocity, phase function and temperature that cannot be eliminated (the Boussinesq
approximation is a lower-order term and is easier to handle). In order to overcome this difficulty, we introduce the smallness assumption \eqref{ini} and try to seek certain energy dissipative {\rm inequality} instead.
 \er

 \br It is worth mentioning that the conditions in
\eqref{ini} does not involve the interfacial parameter
$\varepsilon$.
 \er

\subsection{Proof of Theorem \ref{theorem on weak existence}.} The proof
is based on a semi-Galerkin method (see, for instance, Lin and Liu
\cite{LL95} for a simplified nematic liquid crystal system). Let $\{w_i\}$ ($i\in \mathbb{N}$) be an orthonormal
basis of $\mathbf{V}$ formed by the eigenvectors of the Stokes
problem
 $$ -\Delta w_i+\nabla P_i=\lambda_i w_i, \ \ \mbox{in} \ \Omega, \quad  w_i|_\Gamma=0, $$
with $\|w_i\|=1$ and $0 < \lambda_1 \leq \lambda_2 \leq \cdots \leq
\lambda_n \leq \cdots$ with $\lambda_n \rightarrow \infty$ as $n
\rightarrow +\infty$.

For every $m\in \mathbb{N}$, let $\mathbf{V}_m=span\{w_1, w_2,
\cdots, w_m\}$. We denote by ${\rm P}_m: \mathbf{H} \rightarrow
\mathbf{V}_m$ the orthogonal projection. Given $(u_0, \phi_0,
\theta_0)\in \mathbf{H} \times (H^1(\Omega)\cap
L^\infty(\Omega))\times (H^1_0(\Omega)\cap L^\infty(\Omega))$
satisfying \eqref{ini}, we consider the following approximate
problem: \bea
 && \langle \partial_tu_{m}, v_m\rangle_{\mathbf{V}', \mathbf{V}}+ \int_\Omega (u_m\cdot\nabla)u_m\cdot v_m dx
 +\nu\int_\Omega \nabla{u_m} : \nabla{v_m} dx \non\\
 && \quad = \int_\Omega [\lambda(\theta_m)\nabla\phi_m\otimes\nabla\phi_m] : \nabla v_m dx+\alpha g\int_\Omega  \theta_m \mathbf{j} \cdot v_m dx,\quad \forall\,
 v_m\in \mathbf{V}_m, \label{nsapp}\\
 && \partial_t\phi_{m}+ u_m\cdot \nabla\phi_m=\gamma (\Delta \phi_m- F'(\phi_m)), \quad \text{a.e. in\ } \Omega, \label{papp}\\
 &&  \partial_t\theta_{m}+ u_m\cdot \nabla\theta_m=k\Delta\theta_m,\quad \text{a.e. in\ } \Omega,\label{tapp}\\
 &&\phi_m(x, t)=-1, \ \ \theta_m(x, t)=0 \ \ \mbox{on} \
\Gamma, \label{BC-App}\\
 &&u_m|_{t=0}={\rm P}_mu_0(x), \ \ \phi_m|_{t=0}=\phi_0(x), \ \ \theta_m|_{t=0}=\theta_0(x). \label{IC-App}
 \eea
 Indeed, we observe that all the {\it a priori} bounds derived (formally) from the energy inequality \eqref{basic energy inequality} still hold for the approximate problem. If we fix $\tilde{u}_m \in C([0,T]; \mathbf{V}_m)$, then we can find
$\phi_m = \phi_m[\tilde{u}_m]$  and $\theta_m= \theta_m[\tilde{u}_m]$ solving \eqref{papp} and \eqref{tapp} (with $u_m=\tilde{u}_m$), respectively. Inserting $\phi_m$ and $\theta_m$ into the equation \eqref{nsapp}, we can find a solution $u_m={\cal T}[\tilde{u}_m]$ that defines a mapping $\tilde{u}_m \mapsto {\cal T}[\tilde{u}_m]$.
On account of the {\it a priori} bounds, we can easily show that ${\cal T}$ admits a fixed point
by means of the classical Schauder's argument on $(0,T_0)$, with $0< T_0 \leq T$.
Finally, applying again the {\it a priori} estimates, we are allowed to conclude that the approximate solutions can be
extended to the whole time interval $[0,+\infty)$ (cf. also
Ezquerra et al \cite[Appendix]{C09}). Since the {\it a priori} estimates of the approximate solution are uniform in  parameter $m$, then using a similar argument as in \cite[Section 2]{LL95}, we can pass to the limit $m\to +\infty$ and complete the proof of Theorem \ref{theorem on
weak existence}. The details are omitted here.

\begin{corollary}\label{low-estimate}
Suppose $n=2, 3$. Under the assumptions of Theorem \ref{theorem on
weak existence}, the weak solution $(u, \phi, \theta)$ to the
problem \eqref{navier-stokes}--\eqref{BC} satisfies \be
\|\phi(t)\|_{L^\infty}\leq 1,\quad \|\theta(t)\|_{L^\infty}\leq
\frac{1}{4C_1^2|b|}\sqrt{\frac{a\gamma\nu}{2\lambda_0}},\quad
\forall\, t\geq 0,\non
 \ee
and the energy inequality
 \be
\mathcal{E}(t)+ \int_0^{+\infty} \left(\frac{\nu}{2}\|\nabla
u\|^2+a\lambda_0\gamma\|\Delta\phi-F'(\phi)\|^2+k\zeta\|\Delta\theta\|^2\right)
dt\leq \mathcal{E}(0),\quad \forall\, t\geq 0,\non
 \ee
which yields the following uniform estimates:
 \bea
 && \|u(t)\|^2+\|\phi(t)\|_{H^1}^2+\|\theta(t)\|_{H^1}^2\leq M, \quad \forall\, t\geq 0,\non \\
 && \int_{0}^{+\infty}\big(\|\nabla u(t)\|^2+\|\Delta\phi(t)-F'(\phi(t))\|^2+\|\Delta\theta(t)\|^2 \big)dt \leq M,\non
 \eea
 where $M > 0$ is a constant depending on $\|u_0\|, \|\phi_0\|_{H^1}, \|\theta_0\|_{H^1}$, $\Omega$, and coefficients of the
 system.
\end{corollary}

%
\section{Strong solutions }
\setcounter{equation}{0}

In this section, we prove the existence and uniqueness of strong
solutions to the problem \eqref{navier-stokes}--\eqref{BC}.
\begin{definition} \label{def of strong solution}
For any $T \in (0, +\infty]$, $u_0\in \mathbf{V}, \phi_0 \in
H^2(\Omega), \theta_0\in H^2(\Omega)\cap H^1_0(\Omega)$, we say that
the triple $(u, \phi, \theta)$ is a strong solution to the problem
\eqref{navier-stokes}--\eqref{BC}, if $(u, \phi, \theta)$ is a weak
solution and
 \be  u \in L^\infty(0, T; \mathbf{V}) \cap
L^2(0, T; \mathbf{H}^2), \quad \phi, \theta\in L^\infty(0, T; {H}^2)
\cap L^2(0, T; {H}^3).\non
 \ee
\end{definition}

Based on the semi-Galerkin scheme in the previous section, in order
to prove the existence of strong solutions, it suffices to derive
proper uniform higher-order estimates for the approximate solutions
and then pass to the limit $m\to +\infty$. We observe that the
approximate solutions satisfy the same basic energy inequality and
higher-order differential inequalities as smooth solutions of the
problem \eqref{navier-stokes}--\eqref{BC}. Thus, for the sake of
simplicity, all the calculations below will be carried out formally
for smooth solutions.

The main results of this section are as follows.

\begin{theorem} [Global strong solution in $2D$] \label{strong2D}
Suppose $n=2$. For any $u_0\in \mathbf{V}, \phi_0 \in H^2(\Omega),
\theta_0\in \big(H^1_0(\Omega) \cap H^2(\Omega)\big)$ satisfying the
assumption \eqref{ini}, the problem
\eqref{navier-stokes}--\eqref{BC} admits a unique global strong
solution such that
 \bea &&u \in {L}^{\infty}(0, +\infty;
\mathbf{V})\cap {L}^2_{loc}(0, +\infty; \mathbf{H}^2), \non\\
 &&\phi \in {L}^{\infty}(0, +\infty; {H}^2)\cap
{L}^2_{loc}(0, +\infty; {H}^3),    \non\\
&&\theta \in {L}^{\infty}(0, +\infty; {H}^2\cap H^1_0)\cap
{L}^2_{loc}(0, +\infty; {H}^3).\non
 \eea
\end{theorem}

\begin{theorem} [Local strong solution in $3D$] \label{locstrong}
Suppose $n=3$.  For any $u_0\in \mathbf{V}, \phi_0 \in H^2(\Omega),
\theta_0\in \big(H^1_0(\Omega) \cap H^2(\Omega)\big)$ satisfying the
assumption \eqref{ini}, the problem
\eqref{navier-stokes}--\eqref{BC} admits a unique local strong
solution.
\end{theorem}

\begin{theorem} [Global strong solution in $3D$ under large viscosity] \label{theorem on large viscosity case}
Suppose $n=3$.  For any $u_0\in \mathbf{V}, \phi_0 \in H^2(\Omega),
\theta_0\in \big(H^1_0(\Omega) \cap H^2(\Omega)\big)$ satisfying the
assumption \eqref{ini}, if in addition, the lower bound of the
viscosity, i.e., $\nu$ is sufficiently large (cf. \eqref{large nu}),
then the problem \eqref{navier-stokes}--\eqref{BC} admits a unique
global strong solution.
\end{theorem}

\subsection{Two dimensional case}
First, we are going to derive a specific type of higher-order energy
inequality in the spirit of \cite{LL95}.

\begin{lemma}\label{high2d}
Suppose $n=2$. Let $(u, \phi, \theta)$ be a smooth solution to the
problem \eqref{navier-stokes}--\eqref{BC}. We introduce the quantity
 \be \mathcal{A}_1(t)= \|\nabla
 u(t)\|^2+a\lambda_0\|\Delta\phi(t)-F'(\phi(t))\|^2+\eta_1\|\Delta\theta(t)\|^2.
 \label{def of higher order energy}
 \ee
 Then the following differential inequality holds:
 \be
 \frac{d}{dt}\mathcal{A}_1(t)+\nu\|S u\|^2+a\lambda_0\gamma\|\nabla(\Delta\phi-F'(\phi))\|^2
 +\eta_1 k\|\nabla\Delta\theta\|^2
 \leq C(\mathcal{A}_1^2(t)+\mathcal{A}_1(t)), \quad \forall\, t > 0.
 \label{higher energy inequality}
 \ee
$\eta_1$ and $C$ are two positive constants which may depend on
$\|u_0\|$, $\|\phi_0\|_{H^1}$, $\|\theta_0\|_{H^1}$, $\Omega$, and
coefficients of the system.
\end{lemma}
\begin{proof}
We observe that $-(\Delta u, u_t)=(Su, u_t)$, since $u_t \in
\mathbf{H}$. Besides, due to \eqref{BC2}, we infer that $\phi_t|_\Gamma =0$, then by \eqref{BC1} and the equation \eqref{phase}, we see that $\Delta\phi-F'(\phi)\big|_{\Gamma}=\phi_t+u\cdot \nabla \phi|_\Gamma=0$. Using equations
\eqref{navier-stokes} and \eqref{phase}, we compute that
 \bea &&\frac12\frac{d}{dt}\big(\|\nabla
u\|^2+a\lambda_0\|\Delta\phi-F'(\phi)\|^2\big)+\nu\|S{u}\|^2+a\lambda_0\gamma\|\nabla(\Delta\phi-F'(\phi))\|^2  \non\\
&=&-(u\cdot\nabla u, S{u})+ b\lambda_0(\nabla\cdot (\theta \nabla
\phi\otimes \nabla \phi), S u)
+\alpha g(\theta \mathbf{j}, S{u})\non\\
&& -a\lambda_0\big( \nabla\phi(\Delta\phi-F'(\phi)), \nabla\pi\big) -a\lambda_0\gamma(F''(\phi)(\Delta\phi-F'(\phi)),
\Delta\phi-F'(\phi))\non\\
 && -2a\lambda_0 \int_{\Omega}(\Delta\phi-F'(\phi))
\nabla_j
u_i\nabla_j\nabla_i\phi \,dx \non\\
&=&\sum_{m=1}^6 K_m. \label{time derivative of A-part1}
 \eea
In what follows, we proceed to estimate the right-hand side of
\eqref{time derivative of A-part1} by using the uniform estimates
obtained in Corollary \ref{low-estimate} and properties of the
Stokes operator.
 \bea K_1 &\leq&
\frac{\nu}{32}\|S u\|^2+C\|u\|_{\mathbf{L}^\infty}^2\|\nabla u\|^2\non\\
&\leq& \frac{\nu}{32}\|S u\|^2+C\|u\|\|\Delta u\|\|\nabla u\|^2\non\\
& \leq&
\frac{\nu}{16}\|S u\|^2+C\|\nabla u\|^4. \non
 \eea
For the second term $K_2$, we have
 \bea
 K_2&=&b\lambda_0\int_\Omega \nabla_j\theta \nabla_i\phi\nabla_j\phi (Su)_i dx
 +\frac{b\lambda_0}{2}(\theta\nabla |\nabla\phi|^2, Su)+b\lambda_0(\theta\Delta\phi\nabla \phi, Su)\non\\
 &=&b\lambda_0\int_\Omega \nabla_j\theta \nabla_i\phi\nabla_j\phi (Su)_i dx-\frac{b\lambda_0}{2}(|\nabla\phi|^2\nabla \theta, Su)
 +b\lambda_0(\theta\Delta\phi\nabla \phi, Su)\non\\
 &:=&K_{2a}+K_{2b}+K_{2c}.\non
 \eea
 \bea
 && K_{2a}+K_{2b}\non\\
 &\leq&
C\|\nabla\theta\|_{\mathbf{L}^6}\|\nabla\phi\|_{\mathbf{L}^6}^2\|S
u\|
\non\\
&\leq&\frac{\nu}{16}\|S
u\|^2+C\big(\|\nabla\Delta\theta\|^{\frac23}\|\nabla\theta\|^{\frac43}+\|\nabla\theta\|^2\big)
\big(\|\nabla\Delta\phi\|^{\frac43}\|\nabla\phi\|^{\frac83}+\|\nabla\phi\|^4\big) \non\\
&\leq&\frac{\nu}{16}\|S
u\|^2+C\big(\|\nabla\Delta\theta\|^{\frac23}\|\nabla\theta\|^{\frac43}+\|\nabla\theta\|^2\big)
(\|\nabla(\Delta\phi-F'(\phi))\|^{\frac43}+\|F''(\phi)\|_{L^\infty}^\frac43\|\nabla\phi\|^\frac43+C)  \non\\
&\leq&\frac{\nu}{16}\|S
u\|^2+\frac{a\lambda_0\gamma}{8}\|\nabla(\Delta\phi-F'(\phi))\|^2+C_5\|\nabla\Delta\theta\|^2+C(\|\nabla
\theta\|^2+\|\nabla \theta\|^6),\non
 \eea
 where $C_5$ is a constant depending on $\|u_0\|, \|\phi_0\|_{H^1}, \|\theta_0\|_{H^1}$, $\Omega$ and coefficients of the
 system.
 \bea
 K_{2c}&\leq&|b|\lambda_0\|\theta\|_{L^\infty}\|\Delta\phi\|\|\nabla\phi\|_{\mathbf{L}^\infty}\|S
u\|   \non\\
&\leq& C\|\theta\|_{L^\infty}\|\Delta\phi\|\|\nabla
\phi\|_{\mathbf{H}^2}^\frac12\|\nabla \phi\|^\frac12\|S
u\|\non\\
&\leq&\frac{\nu}{16}\|S
u\|^2+C\|\theta\|_{L^\infty}^2(\|\Delta\phi-F'(\phi)\|^2+C)(\|\nabla(\Delta\phi-F'(\phi))\|+\|\Delta\phi-F'(\phi)\|+C)
\non\\
&\leq& \frac{\nu}{16}\|S
u\|^2+C\|\Delta\phi-F'(\phi)\|^2\|\nabla(\Delta\phi-F'(\phi))\|\non\\
&& + C\|\theta\|\|\Delta \theta\|(\|\nabla(\Delta\phi-F'(\phi))\|+\|\Delta\phi-F'(\phi)\|^3+C)\non\\
&\leq&\frac{\nu}{16}\|S
u\|^2+\frac{a\lambda_0\gamma}{8}\|\nabla(\Delta\phi-F'(\phi))\|^2\non\\
&&+C(\|\Delta\phi-F'(\phi)\|^2+\|\Delta\phi-F'(\phi)\|^4+\|\Delta\theta\|^2+\|\Delta\theta\|^4).\non
 \eea
 The remaining terms can be estimated as follows
 \be K_3 \leq |\alpha||g|\|\theta\|\|S u\| \leq \frac{\nu}{16}\|S
u\|^2+C\|\Delta\theta\|^2. \non
 \ee
\bea K_4&\leq&
a\lambda_0\|\nabla\pi\|\|\nabla\phi\|_{\mathbf{L}^4}\|\Delta\phi-F'(\phi)\|_{L^4}
\non\\
&\leq&C\|S
u\|\big(\|\Delta\phi-F'(\phi)\|^{\frac12}+C\big)\|\Delta\phi-F'(\phi)\|^{\frac12}\|\nabla(\Delta\phi-F'(\phi))\|^{\frac12}\non
\\
&\leq&\frac{\nu}{16}\|S
u\|^2+\frac{a\lambda_0\gamma}{8}\|\nabla(\Delta\phi-F'(\phi))\|^2+C(\|\Delta\phi-F'(\phi)\|^2+\|\Delta\phi-F'(\phi)\|^4).\non
 \eea
  \be
 K_5\leq  a\lambda_0\gamma \|F''(\phi)\|_{L^\infty}\|\Delta\phi-F'(\phi)\|^2
\leq C\|\Delta\phi-F'(\phi)\|^2. \non
 \ee
\bea
K_6&=&2a\lambda_0\int_{\Omega}\nabla_i(\Delta\phi-F'(\phi))\nabla_j
u_i\nabla_j\phi\,dx   \non\\
&\leq&C\|\nabla(\Delta\phi-F'(\phi))\|\|\nabla
u\|_{\mathbf{L}^4}\|\nabla\phi\|_{\mathbf{L}^4}  \non\\
&\leq&\frac{a\lambda_0\gamma}{12}\|\nabla(\Delta\phi-F'(\phi))\|^2+C\|\nabla
u\|\|\Delta u\|\big(\|\Delta\phi-F'(\phi)\|+\|F'(\phi)\|
\big)\|\nabla\phi\| \non\\ &\leq&\frac{\nu}{16}\|S
u\|^2+\frac{a\lambda_0\gamma}{8}\|\nabla(\Delta\phi-F'(\phi))\|^2+C(\|\nabla
u\|^2+\|\nabla u\|^4+\|\Delta\phi-F'(\phi)\|^4). \non
 \eea
It follows from the above estimates and the Sobolev embedding that
\bea &&\frac{1}{2}\frac{d}{dt}(\|\nabla
u\|^2+a\lambda_0\|\Delta\phi-F'(\phi)\|^2 )+\frac{5\nu}{8}\|S
u\|^2+\frac{a\lambda_0\gamma}{2}\|\nabla(\Delta\phi-F'(\phi))\|^2  \non\\
&\leq& C_5
\|\nabla\Delta\theta\|^2+C(\|\nabla u\|^4+\|\Delta\phi-F'(\phi)\|^4+\|\Delta \theta\|^4)\non\\
&& +C(\|\Delta u\|^2+\|\Delta\phi-F'(\phi)\|^2+\|\Delta \theta\|^2).
\label{part 1 of high energy inequality}
  \eea
We infer from \eqref{BC} that $\theta_t|_\Gamma=0$, then it follows from \eqref{BC1} and the $\theta$-equation \eqref{temperature} that
$\Delta\theta|_{\Gamma}=0$. Applying $\Delta$ to both sides of \eqref{temperature}, and taking the $L^2$-inner product of the resultant with $\Delta\theta$, we obtain
 \be
\frac12\frac{d}{dt}\|\Delta\theta\|^2 +k\|\nabla\Delta\theta\|^2
 =-\int_{\Omega}\Delta(u\cdot\nabla\theta)\Delta\theta\,dx
:=K_7, \label{temperature-1}
 \ee
 such that
 \bea
K_7 &=&
\int_{\Omega}\nabla(u\cdot\nabla\theta)\cdot\nabla\Delta\theta\,dx
\non\\
&\leq&\frac{k}{8}\|\nabla\Delta\theta\|^2+\frac{2}{k}\|\nabla(u\cdot\nabla\theta)\|^2
\non\\
&\leq&\frac{k}{8}\|\nabla\Delta\theta\|^2+C\big(\|\nabla
u\|_{\mathbf{L}^4}^2\|\nabla\theta\|_{\mathbf{L}^4}^2+\|u\|_{\mathbf{L}^\infty}^2\|\Delta\theta\|^2\big)\non\\
&\leq& \frac{k}{8}\|\nabla\Delta\theta\|^2+ C\|\Delta
u\|^{\frac32}\|u\|^{\frac12}(\|\Delta\theta\|\|\nabla\theta\|+\|\nabla\theta\|^2)+
C\|\Delta
u\|\|u\|\|\Delta\theta\|^2\non\\
&\leq& \frac{k}{4}\|\nabla\Delta\theta\|^2+ \frac{\nu}{8kC_5}\|S
u\|^2+C\|\nabla{u}\|^2+ C\|\Delta\theta\|^4.\label{K_9}
 \eea
Hence, multiplying \eqref{temperature-1} by $\eta_1=kC_5$, and
adding the result to \eqref{part 1 of high energy inequality}, we
deduce our conclusion \eqref{higher energy inequality} from
\eqref{K_9}. The proof is complete.
\end{proof}

\textbf{Proof of Theorem \ref{strong2D}}. Since $u_0\in \mathbf{V}, \phi_0 \in H^2(\Omega), \theta_0\in \big(H^1_0(\Omega) \cap H^2(\Omega)\big)$, we have $\mathcal{A}_1(0)<+\infty$.
It follows from Corollary \ref{low-estimate} that
$\mathcal{A}_1(t)\in L^1(0, +\infty)$. Then we infer from Lemma \ref{high2d} and \cite[Lemma 6.2.1]{Z04} that $\mathcal{A}_1(t)$ is uniformly bounded for all time, which implies
  \be
 \|u(t)\|_{\mathbf{H}^1}+ \|\phi(t)\|_{H^2}+ \|\theta(t)\|_{H^2}+ \int_t^{t+1}(\| u(s)\|_{\mathbf{H}^2}^2+\|\phi(s)\|_{H^3}^2+\|\theta(s)\|_{H^3}^2) ds\leq C, \ \forall\, t\geq 0. \label{uniform boud}
 \ee
 where $C$ is a positive constant depending on $\|u_0\|_{\mathbf{H}^1}$, $\|\phi_0\|_{H^2}$, $\|\theta_0\|_{H^2}$, $\Omega$, and
 coefficients of the system.
 Then we can prove the existence of a global strong solution, which is actually unique
 by Proposition \ref{wsuniq}. The proof is complete.  \quad \quad \quad $\square$

For the weak solution, we still have $\mathcal{A}_1(t)\in L^1(0, +\infty)$. Then by
\eqref{higher energy inequality} and the uniform Gronwall lemma
(cf. Temam \cite[Lemma III.1.1]{Te97}), we conclude that for any $\delta>0$,
 \be
 \mathcal{A}_1(t+\delta)\leq C\left(1+\frac1\delta\right), \quad \forall\, t\geq
 0,\non
 \ee
 where $C$ is a positive constant depending on $\|u_0\|$, $\|\phi_0\|_{H^1}$, $\|\theta_0\|_{H^1}$, $\Omega$, and
 coefficients of the system. As a result,
 \begin{proposition}[Regularity of weak solutions in $2D$] \label{2dreg}
 When $n=2$, under the assumptions of Theorem \ref{theorem on weak existence}, any weak solution to problem \eqref{navier-stokes}--\eqref{BC} becomes a strong one for $t>0$ and the following estimate holds
 \be
 \|u(t)\|_{\mathbf{H}^1}^2+ \|\phi(t)\|_{H^2}^2+ \|\theta(t)\|_{H^2}^2 +\int_t^{t+1}(\|u(s)\|_{\mathbf{H}^2}^2+\|\phi(s)\|_{H^3}^2+\|\theta(s)\|_{H^3}^2) ds\leq \mathcal{D}(t), \  \forall\, t>0,\non
 \ee
  where $\mathcal{D}$ is a positive function depending on $\|u_0\|$, $\|\phi_0\|_{H^1}$, $\|\theta_0\|_{H^1}$, $\Omega$, and coefficients of the
  system. In particular, $\displaystyle\lim_{t\to 0^+}\mathcal{D}(t)=+\infty$.
 \end{proposition}

\subsection{Three dimensional case}
\begin{lemma}\label{high3da}
Suppose $n=3$. Let $(u, \phi, \theta)$ be a smooth solution to
problem \eqref{navier-stokes}--\eqref{BC}. For the quantity
$\mathcal{A}_2(t)$
 \be \mathcal{A}_2(t)= \|\nabla
u(t)\|^2+a\lambda_0\|\Delta\phi(t)-F'(\phi(t))\|^2+\|\Delta\theta(t)\|^2,
 \label{def of higher order energy3d}
 \ee
the following differential inequality holds
 \be
 \frac{d}{dt}\mathcal{A}_2(t)+\nu\|S u\|^2+a\lambda_0\gamma\|\nabla(\Delta\phi-F'(\phi))\|^2+k\|\nabla\Delta\theta\|^2
 \leq C_*(\mathcal{A}_2^4(t)+\mathcal{A}_2(t)), \quad \forall\, t > 0.
 \label{higher energy inequality1}
 \ee
 $C_*$ is a positive constant
which may depend on $\|u_0\|$, $\|\phi_0\|_{H^1}$,
$\|\theta_0\|_{H^1}$, $\Omega$, and coefficients of the system.
\end{lemma}
 \br We note that the coefficient of the third term in
$\mathcal{A}_2(t)$ is different from the one in $\mathcal{A}_1(t)$ (see \eqref{def of higher order energy}).
 \er
\begin{proof}
 We re-estimate the right-hand side of \eqref{time derivative of
A-part1} and \eqref{temperature-1} by using the $3D$ version of Sobolev
embedding theorems. The estimates for $K_3$ and $K_5$ remain
unchanged. Next,
\be K_1 \leq  \|u\|_{\mathbf{L}^\infty} \|\nabla u\|\|S u\|\leq C\|S
u\|^\frac32\|\nabla u\|^\frac32\leq \frac{\nu}{12}\|S
u\|^2+C\|\nabla u\|^6.\non \ee \bea K_{2a}+K_{2b}&\leq &
C\|\nabla\theta\|_{\mathbf{L}^6}\|\nabla\phi\|_{\mathbf{L}^6}^2\|S
u\|
\leq\frac{\nu}{12}\|S u\|^2+ C\|\Delta \theta\|^2\|\Delta \phi\|^4\non\\
&\leq& \frac{\nu}{12}\|S u\|^2+ C\|\Delta \theta\|^2\|\Delta
\phi-F'(\phi)\|^4+C\|\Delta \theta\|^2.\non
 \eea
 \bea
 K_{2c}&\leq&|b|\lambda_0\|\theta\|_{L^\infty}\|\Delta\phi\|\|\nabla\phi\|_{\mathbf{L}^\infty}\|S
u\|   \non\\
&\leq& C\|\Delta \theta\|^\frac12\|\nabla
\theta\|^\frac12\|\Delta\phi\|\|\nabla
\phi\|_{\mathbf{H}^2}^\frac12\|\Delta \phi\|^\frac12\|S
u\|\non\\
&\leq&\frac{\nu}{12}\|S u\|^2+C\|\Delta \theta\|\|\nabla
\theta\|(\|\Delta\phi-F'(\phi)\|^3+1)(\|\nabla(\Delta\phi-F'(\phi))\|+\|\Delta\phi-F'(\phi)\|+1)
\non\\
&\leq&\frac{\nu}{12}\|S
u\|^2+\frac{a\lambda_0\gamma}{6}\|\nabla(\Delta\phi-F'(\phi))\|^2\non\\
&&+C(\|\Delta\phi-F'(\phi)\|^2+\|\Delta\phi-F'(\phi)\|^8+\|\Delta\theta\|^2+\|\Delta\theta\|^8).\non
 \eea
 \bea
K_4&\leq&
a\lambda_0\|\nabla\pi\|\|\nabla\phi\|_{\mathbf{L}^6}\|\Delta\phi-F'(\phi)\|_{{L}^3}
\non\\
&\leq&C\|S
u\|\big(\|\Delta\phi-F'(\phi)\|+1\big)\|\Delta\phi-F'(\phi)\|^{\frac12}\|\nabla(\Delta\phi-F'(\phi))\|^{\frac12}\non
\\
&\leq&\frac{\nu}{12}\|S
u\|^2+\frac{a\lambda_0\gamma}{6}\|\nabla(\Delta\phi-F'(\phi))\|^2+C(\|\Delta\phi-F'(\phi)\|^2+\|\Delta\phi-F'(\phi)\|^6).\non
 \eea
 \bea
K_6
&=&2a\lambda_0\int_{\Omega}\nabla_j(\Delta\phi-F'(\phi))u_i\nabla_j\nabla_i\phi\,dx
-2a\lambda_0\big(\nabla(\Delta\phi-F'(\phi)), F'(\phi)u \big)   \non\\
&\leq&C\|\nabla(\Delta\phi-F'(\phi))\|\big(\|
u\|_{\mathbf{L}^6}\|\nabla^2\phi\|_{\mathbf{L}^3}+\|F'(\phi)\|_{\mathbf{L}^\infty}\|u\|  \big)  \non\\
&\leq&C\|\nabla(\Delta\phi-F'(\phi))\|\|\nabla{u}\|\big(\|\Delta\phi-F'(\phi)\|^{\frac12}\|\nabla(\Delta\phi-F')\|^{\frac12}
+\|\Delta\phi-F'(\phi)\|+1\big)
  \non\\
&\leq&\frac{a\lambda_0\gamma}{6}\|\nabla(\Delta\phi-F'(\phi))\|^2+C(\|\nabla
u\|^2+\|\nabla
u\|^8+\|\Delta\phi-F'(\phi)\|^2+\|\Delta\phi-F'(\phi)\|^8). \non
 \eea
 \bea
K_{7}&\leq&
\frac{k}{2}\|\nabla\Delta\theta\|^2+\frac{1}{k}\|\nabla(u\cdot\nabla\theta)\|^2
\non\\
&\leq&\frac{k}{2}\|\nabla\Delta\theta\|^2+C\big(\|\nabla
u\|_{\mathbf{L}^4}^2\|\nabla\theta\|_{\mathbf{L}^4}^2+\|u\|_{\mathbf{L}^\infty}^2\|\Delta\theta\|^2\big)
\non\\
&\leq&\frac{k}{2}\|\nabla\Delta\theta\|^2+C\|S
u\|^{\frac32}\|\nabla u\|^{\frac12}\|\Delta\theta\|^2  \non\\
&\leq&\frac{k}{2}\|\nabla\Delta\theta\|^2+\frac{\nu}{12}\|S
u\|^2+C\|\nabla u\|^2\|\Delta\theta\|^4.\non
 \eea
 Collecting all the estimates above, we arrive at our conclusion \eqref{higher energy inequality1}.
 \end{proof}
 \textbf{Proof of  Theorem \ref{locstrong}}. Due to Lemma \ref{high3da}, a standard argument of the ordinary differential equation yields that there is a time $T_0=T_0(u_0, \phi_0, \theta_0)<+\infty$ such that $\mathcal{A}_2(t)$ is bounded on $[0,T_0]$. This enables us to prove that problem \eqref{navier-stokes}--\eqref{BC} admits a local strong solution. Uniqueness of the strong solution follows from Proposition \ref{wsuniq}. \quad \quad \quad $\square$

Since our problem contains the Navier--Stokes equation as a
subsystem, in the $3D$ case, we cannot except the existence of
global strong solutions to problem \eqref{navier-stokes}--\eqref{BC}
for arbitrary large initial data. However, the global strong solution will
exist if we further assume that the lower bound of the viscosity
$\nu^*$ is sufficiently large.

Set $$\tilde{\mathcal{A}}_2(t)= \mathcal{A}_2(t)+1.$$
Then Theorem \ref{theorem on large viscosity case} is a direct
consequence of the following higher-order differential inequality
concerning $\tilde{\mathcal{A}}_2(t)$.

\begin{lemma} \label{proposition on high energy inequality 3D-1}
Suppose $n=3$. We assume that $\nu \geq 1$ and \eqref{ini} is
fulfilled. Let $(u, \phi, \theta)$ be a smooth solution to the
problem \eqref{navier-stokes}--\eqref{BC}. Then the following
inequality holds
 \bea && \frac{d}{dt}\tilde{\mathcal{A}}_2(t)+\big[\nu-M_1\nu^\frac12\tilde{\mathcal{A}}_2(t) \big]\|S
u\|^2+\Big(a\lambda_0\gamma-\frac{M_1\tilde{\mathcal{A}}_2(t)}{\nu^\frac12}
\Big)\|\nabla(\Delta\phi-F'(\phi))\|^2+k\|\nabla\Delta\theta\|^2\non\\
&\leq& M_2\mathcal{A}_2(t), \label{high energy inequality 3D-2}
 \eea
 $M_1$ and $M_2$ are constants depending on $\|u_0\|$, $\|\phi_0\|_{H^1}$,
 $\|\theta_0\|_{H^1}$, $\Omega$, and coefficients of the system, but not on $\nu$.
\end{lemma}
\begin{proof} We note that the uniform estimates in Corollary \ref{low-estimate} still hold. Then we re-estimate the terms $K_1,..., K_{7}$
in an alternative way. The estimates for $K_3$ and $K_5$ remain
unchanged. For the other terms, we have
 \bea K_1 &\leq&  C\|S
u\|^{\frac74}\|u\|^\frac14\|\nabla u\|\non\\
& \leq&  \frac{\nu}{12}\|S
u\|^2+\nu^\frac12\|\nabla u\|^2\|\Delta
u\|^2+C\nu^{-\frac{11}{2}}\|\nabla u\|^2,\non
 \eea
 \bea &&
K_{2a}+K_{2b}\leq
C\|\nabla\theta\|_{\mathbf{L}^6}\|\nabla\phi\|_{\mathbf{L}^6}^2\|S u\|\non\\
&\leq&C\|\nabla\Delta\theta\|^{\frac12}\|\nabla\theta\|^{\frac12}
(\|\Delta\phi-F'(\phi)\|+1)\big(\|\nabla(\Delta\phi-F'(\phi))\|^{\frac12}\|\nabla\phi\|^{\frac12}+1
\big)\|S u\|    \non\\
&\leq&
C\|\nabla\Delta\theta\|^{\frac12}\|\nabla\theta\|^{\frac12}(\|\Delta\phi-F'(\phi)\|+1)\|S
u\|  \non\\
&&+C\|\nabla\Delta\theta\|^\frac12\|\nabla\theta\|^{\frac12}(\|\Delta\phi-F'(\phi)\|+1)\|S
u\|\|\nabla(\Delta\phi-F'(\phi))\|^{\frac12} \non\\
&\leq&
\Big[\frac{\nu}{12}+\nu^\frac12(\|\Delta\phi-F'(\phi)\|+\|\nabla
\theta\|) \Big]\|S
u\|^2 +\frac{k}{4}\|\nabla\Delta\theta\|^2\non\\
&&
+\frac{C}{\nu}(1+\|\Delta\phi-F'(\phi)\|^2)\|\nabla(\Delta\phi-F'(\phi))\|^2\non\\
&& +
C\Big(\frac{1}{\nu}+\frac{1}{\nu^2}\Big)(\|\nabla
\theta\|^2+\|\Delta\phi-F'(\phi)\|^2), \non
 \eea
 \bea
 K_{2c}&\leq&|b|\lambda_0\|\theta\|_{L^\infty}\|\Delta\phi\|\|\nabla\phi\|_{\mathbf{L}^\infty}\|S
u\|   \non\\
&\leq& C\|\Delta \theta\|^\frac12\|\nabla
\theta\|^\frac12\|\Delta\phi\|^\frac32(\|\nabla
\Delta\phi\|^\frac12+\|\Delta\phi\|^\frac12)\|S
u\|\non\\
&\leq&  C\|\Delta \theta\|^\frac12\|\nabla
\theta\|^\frac12(\|\Delta\phi-F'(\phi)\|^\frac32+1)(\|\nabla
(\Delta\phi-F'(\phi))\|^\frac12+ 1)\|S
u\|\non\\
&\leq&
\Big[\frac{\nu}{12}+\nu^\frac12(\|\Delta\phi-F'(\phi)\|^2+\|\Delta
\theta\|) \Big]\|S
u\|^2 \non\\
&&+\Big(\frac{a\lambda_0\gamma}{6}+\frac{1}{\nu}\|\Delta\phi-F'(\phi)\|^2\Big)\|\nabla(\Delta\phi-F'(\phi))\|^2\non\\
&& +C\left(1+\frac{1}{\nu}\right)(\|\Delta
\theta\|^2+\|\Delta\phi-F'(\phi)\|^2),\non
 \eea
 \bea
K_4 &\leq&C\|S
u\|\big(\|\Delta\phi-F'(\phi)\|+1\big)\|\Delta\phi-F'(\phi)\|^{\frac12}\|\nabla(\Delta\phi-F'(\phi))\|^{\frac12}\non
\\
&\leq&\Big(\frac{\nu}{12}+\nu^\frac12\|\Delta\phi-F'(\phi)\|^2
\Big)\|S
u\|^2+\frac{a\lambda_0\gamma}{6}\|\nabla(\Delta\phi-F'(\phi))\|^2+\frac{C}{\nu}\|\Delta\phi-F'(\phi)\|^2,\non
 \eea
 \bea
K_6&\leq&C\|\Delta\phi-F'(\phi)\|_{L^3}\|\nabla
u\|_{\mathbf{L}^6}\|\Delta \phi\|  \non\\
&\leq&C \|\nabla(\Delta\phi-F'(\phi))\|^\frac12\|\Delta\phi-F'(\phi)\|^\frac12\|S u\|\big(\|\Delta\phi-F'(\phi)\|+1 \big) \non\\
&\leq&\left(\frac{\nu}{12}+\nu^\frac12\|\Delta\phi-F'(\phi)\|\right)\|S
u\|^2\non\\
&& +\left(\frac{a\lambda_0\gamma}{6}+\frac{1}{\nu} \|\Delta\phi-F'(\phi)\|^2\right) \|\nabla(\Delta\phi-F'(\phi))\|^2\non\\
&& +C\left(1+\frac{1}{\nu^2}\right)\|\Delta\phi-F'(\phi)\|^2, \non
 \eea
 \bea
K_7&\leq& \frac{k}{2}\|\nabla\Delta\theta\|^2+C\big(\|\nabla
u\|_{\mathbf{L}^4}^2\|\nabla\theta\|_{\mathbf{L}^4}^2+\|u\|_{\mathbf{L}^\infty}^2\|\Delta\theta\|^2\big)
\non\\
&\leq&\frac{k}{2}\|\nabla\Delta\theta\|^2+C\|S{u}\|^{\frac32}\|\nabla{u}\|^{\frac12}\|\Delta\theta\|^\frac32\|\nabla
\theta\|^\frac12+C\|S
u\|^\frac32\| u\|^\frac12\| \Delta\theta\|^2 \non\\
&\leq&\frac{k}{2}\|\nabla\Delta\theta\|^2+\nu^\frac12\|\Delta\theta\|^2
\|S u\|^2+\frac{C}{\nu^\frac32}(\|\nabla u\|^2
+\|\Delta\theta\|^2).\non
 \eea
Combining the above estimates, using the fact $\nu\geq 1$ and the
definition of $\tilde{\mathcal{A}}_2(t)$, we deduce the inequality
\eqref{high energy inequality 3D-2}. The proof is complete.
\end{proof}

\textbf{Proof of Theorem \ref{theorem on large viscosity case}}. It
follows from Corollary \ref{low-estimate} that
 \be
\int_{t}^{t+1}\mathcal{A}_2(\tau)d\tau \leq \tilde{M},\quad \forall
\, t\geq 0,
 \ee where
$\tilde{M} > 0$ may depend on $\Omega, \|u_0\|, \|\phi_0\|_{H^1},
\|\theta_0\|_{H^1}$, and coefficients of the system except $\nu$.
Moreover, if the viscosity $\nu$ is sufficiently large such that
\be \nu^{\frac12} \geq \max\left\{1,
\frac{1}{a\lambda_0\gamma}\right\}
M_1\left(\tilde{\mathcal{A}}_2(0)+M_2\tilde{M}+2\tilde{M}\right)+1,
 \label{large nu}
 \ee
then following the same argument as in \cite{LL95, LWX10}, we can
use Lemma \ref{proposition on high energy inequality 3D-1} to obtain
the uniform estimate
 \be \tilde{\mathcal{A}}_2(t) \leq \frac{\nu^{\frac12}\min\{a\lambda_0\gamma, 1\}}{M_1}, \quad
\forall \, t \geq 0, \label{estimate for tilde A}
 \ee
 which yields
the required conclusion.  \quad $\square$
\subsection{Uniqueness of strong solutions}
The uniqueness of strong solutions to the problem
\eqref{navier-stokes}--\eqref{BC} can be obtained by the energy
method.
\begin{proposition}[Uniqueness of strong solutions] \label{wsuniq}
For $n=2,3$, let $(u_1, \phi_1, \theta_1)$ and $(u_2, \phi_2,
\theta_2)$ be two strong solutions on [0,T] that start from the same
initial data
 $(u_0, \phi_0, \theta_0)\in \mathbf{V}
\times H^2(\Omega) \times \big(H^1_0(\Omega) \cap H^2(\Omega)\big)$
satisfying \eqref{ini}, then $ (u_1, \phi_1, \theta_1)=(u_2, \phi_2,
\theta_2)$.
\end{proposition}
\begin{proof}
We provide the proof for $3D$ case and the proof for $2D$ is
similar. Denote
\[ \bar{u}=u_1-u_2, \ \ \bar{\phi}=\phi_1-\phi_2, \ \ \bar{\theta}=\theta_1-\theta_2. \]
 We can see that $(\bar u, \bar \phi, \bar \theta)$ satisfy
 \bea
 &&\langle \bar u_t, v\rangle_{\mathbf{V}', \mathbf{V}}+
 \int_\Omega [(u_1\cdot\nabla)u_1-(u_2\cdot\nabla)u_2]\cdot v dx+\nu\int_\Omega \nabla\bar{u} :
 \nabla\bar{v}dx \non\\
 &=&\int_\Omega [\lambda(\theta_1)\nabla\phi_1\otimes\nabla\phi_1-\lambda(\theta_2)\nabla\phi_2\otimes\nabla\phi_2] : \nabla v dx
 +\alpha g\int_\Omega  \bar \theta \mathbf{j} \cdot v dx,\, \forall\,
 v\in \mathbf{V},\label{dns}\\
 &&\bar \phi_t+ u_1\cdot \nabla\phi_1- u_2\cdot \nabla\phi_2=\gamma (\Delta \bar\phi- F'(\phi_1)+F'(\phi_2)), \quad \text{a.e. in\ } \Omega, \label{dphi}\\
 &&\bar \theta_t+ u_1\cdot \nabla\theta_1-u_2\cdot\nabla \theta_2=k\Delta\bar{\theta},
 \quad \text{a.e. in\ } \Omega.\label{dthe}
 \eea
Taking  $v=\bar{u}$ in \eqref{dns}, testing \eqref{dphi} by
$-a\lambda_0\Delta\bar{\phi}$ and \eqref{dthe} by
$-\Delta\bar{\theta}$ in $L^2(\Omega)$, respectively, adding up
these three resultants, then performing integration by parts and
using the incompressible condition for the velocity, we get
\bea
&&\frac12\frac{d}{dt}\big(\|\bar{u}\|^2+a\lambda_0\|\nabla\bar{\phi}\|^2+\|\nabla\bar{\theta}\|^2
\big)+\nu\|\nabla\bar{u}\|^2+a\lambda_0\gamma\|\Delta\bar{\phi}\|^2+k\|\Delta\bar{\theta}\|^2
\non\\
&=&-(\bar{u}\cdot\nabla{u_1},
\bar{u})-a\lambda_0(\nabla\bar\phi\Delta \phi_1, \bar
u)+a\lambda_0(u_1\cdot \nabla\bar{\phi},
\Delta\bar{\phi})-b\lambda_0\big(\bar{\theta}\nabla\phi_1\otimes\nabla\phi_1,
\nabla\bar{u}\big) \non\\
&&-b\lambda_0\big(\theta_2\nabla\bar{\phi}\otimes\nabla\phi_1,
\nabla\bar{u}\big)-b\lambda_0\big(\theta_2\nabla\phi_2\otimes\nabla\bar{\phi},
\nabla\bar{u}\big)+\alpha {g}\big(\bar{\theta}\mathbf{j}, \bar{u}
\big)  \non\\
&&+a\lambda_0\gamma\big(F'(\phi_1)-F'(\phi_2), \Delta\bar{\phi}\big)
+(\bar{u}\cdot\nabla\theta_1,
\Delta\bar{\theta})+(u_2\cdot\nabla\bar{\theta},
\Delta\bar{\theta})\non\\
 &:=& \sum_{m=1}^{10}I_m.
\label{eni} \eea
Keeping in mind the uniform estimates obtained in Corollary
\ref{low-estimate}, we proceed to estimate the right hand side of
\eqref{eni},
 \be I_1\leq
\|\bar{u}\|_{\mathbf{L}^4}^2\|\nabla{u_1}\| \leq
C\|\nabla\bar{u}\|^\frac32 \|\bar{u}\|^\frac12\|\nabla u_1\|\leq
\frac{\nu}{12}\|\nabla\bar{u}\|^2+C\|\nabla u_1\|^4\|\bar{u}\|^2.
\non
 \ee
 \bea I_2&\leq& a\lambda_0 \|\nabla\bar{\phi}\|_{\mathbf{L}^6}\|\bar{u}\|_{\mathbf{L}^3}\|\Delta\phi_1\|
 \leq C\|\Delta\phi_1\|\|\bar u\|^\frac12\|\nabla \bar u\|^\frac12\|\Delta \bar \phi\|\non\\
 &\leq& \frac{\nu}{12}\|\nabla \bar u\|^2+ \frac{a\lambda_0\gamma}{10}\|\Delta \bar \phi\|^2+
 C\|\Delta\phi_1\|^4\|\bar u\|^2.\non
 \eea
 \bea I_3 &\leq&
\|u_1\|_{\mathbf{L}^6}\|\Delta\bar{\phi}\|\|\nabla\bar{\phi}\|_{\mathbf{L}^3}
\leq C\|\nabla u_1\|\|\Delta\bar{\phi}\|^\frac32\|\nabla \bar \phi\|^\frac12 \non\\
&\leq&\frac{a\lambda_0\gamma}{10}\|\Delta \bar \phi\|^2 +C\|\nabla
u_1\|^4\|\nabla \bar \phi \|^2.\non \eea
 \bea
I_4+I_7
&\leq&|b|\lambda_0\|\nabla\bar{u}\|\|\bar{\theta}\|_{L^\infty}\|\nabla\phi_1\|_{\mathbf{L}^4}^2+|\alpha||g|\|\bar{\theta}\|\|\bar{u}\|
\non\\
&\leq&C\|\nabla\bar{u}\|\|\nabla\bar{\theta}\|^{\frac12}\|\Delta\bar{\theta}\|^{\frac12}+C\|\nabla\bar{u}\|\|\nabla\bar{\theta}\|  \non\\
&\leq&\frac{\nu}{12}\|\nabla\bar{u}\|^2+\frac{k}{6}\|\Delta\bar{\theta}\|^2+C\|\nabla
\bar{\theta}\|^2.\non \eea
 \bea I_5
&\leq&|b|\lambda_0\|\theta_2\|_{L^\infty}\|\nabla\bar{u}\|\|\nabla\bar{\phi}\|_{\mathbf{L}^3}\|\nabla\phi_1\|_{\mathbf{L}^6}
 \non\\
 &\leq& C\|\nabla\bar{u}\|\|\Delta\bar{\phi}\|^{\frac12}
\|\nabla\bar{\phi}\|^{\frac12}\|\Delta \phi_1\| \non\\
&\leq&\frac{\nu}{12}\|\nabla\bar{u}\|^2+\frac{a\lambda_0\gamma}{10}\|\Delta\bar{\phi}\|^2+C\|\Delta
\phi_1\|^4\|\nabla \bar{\phi}\|^2.\non
 \eea
\bea I_6 &\leq&
|b|\lambda_0\|\theta_2\|_{L^\infty}\|\nabla\bar{u}\|\|\nabla\bar{\phi}\|_{\mathbf{L}^3}\|\nabla\phi_2\|_{\mathbf{L}^6}
\non\\
&\leq&
C\|\nabla\bar{u}\|\|\Delta\bar{\phi}\|^{\frac12}\|\nabla\bar{\phi}\|^{\frac12}
\|\|\Delta{\phi}_2\| \non\\
&\leq&\frac{\nu}{12}\|\nabla\bar{u}\|^2+\frac{a\lambda_0\gamma}{10}\|\Delta\bar{\phi}\|^2+C\|\Delta{\phi}_2\|^4\|\nabla\bar{\phi}\|^2.\non
\eea
 \bea
I_8 &\leq& a\lambda_0\gamma \Big\|\bar\phi\int_0^1 F''(s
\phi_1+(1-s)\phi_2) ds\Big\|_{L^\infty}\|\Delta \bar{\phi}\| \non\\
&\leq&
C\|F''\|_{L^\infty}\|\bar{\phi}\|_{L^\infty}\|\Delta
\bar{\phi}\|\non\\
&\leq& C\|\Delta \bar \phi\|^\frac32 \|\nabla\bar\phi\|^\frac12 \leq
\frac{a\lambda_0\gamma}{10}\|\Delta\bar{\phi}\|^2+C\|\nabla\bar{\phi}\|^2.\non
 \eea
 \bea I_{9} &\leq&
\|\bar{u}\|_{\mathbf{L}^3}\|\nabla\theta_1\|_{\mathbf{L}^6}\|\Delta\bar{\theta}\|
\leq C\|\Delta
\theta_1\|\|\bar{u}\|^{\frac12}\|\nabla\bar{u}\|^{\frac12}\|\Delta\bar{\theta}\|
\non\\
&\leq&\frac{\nu}{12}\|\nabla\bar{u}\|^2+\frac{k}{6}\|\Delta\bar{\theta}\|^2+C\|\Delta
\theta_1\|^4\|\bar{u}\|^2. \non
 \eea
 \bea
I_{10}&\leq&\|u_2\|_{\mathbf{L}^6}\|\nabla\bar{\theta}\|_{\mathbf{L}^3}\|\Delta\bar\theta\|
\leq C\|\nabla{u}_2\|\|\nabla\bar{\theta}\|^{\frac12}\|\Delta\bar{\theta}\|^{\frac32}\non\\
&\leq&\frac{k}{6}\|\Delta\bar{\theta}\|^2+C\|\nabla{u}_2\|^4\|\nabla\bar{\theta}\|^2.\non
 \eea
Summing up, we obtain that
 \bea && \frac{d}{dt}(\|\bar{u}\|^2+a\lambda_0\|\nabla
\bar{\phi}\|^2+\|\nabla \bar{\theta}\|^2)
+\nu\|\nabla\bar{u}\|^2+a\lambda_0\gamma\|\Delta\bar{\phi}\|^2+k\|\Delta\bar{\theta}\|^2
\non\\
&\leq& Q(t)(\|\bar{u}\|^2+a\lambda_0\|\nabla \bar{\phi}\|^2+\|\nabla
\bar{\theta}\|^2), \label{dddf}
 \eea where $$Q(t)=C(1+\|\nabla u_1\|^4+\|\Delta \phi_1\|^4+ \|\Delta \theta_1\|^4+\|\nabla u_2\|^4+\|\Delta \phi_2\|^4)$$
with $C$ being a constant that may depend  on $M$, $\Omega$, and
coefficients of the system. Since the $\mathbf{V}\times H^2\times
H^2$-norm of strong solutions to problem
\eqref{navier-stokes}--\eqref{BC} are bounded on its existence time
interval, $Q(t)$ is bounded on $[0,T]$. Then we can complete the proof
by applying the Gronwall inequality.
\end{proof}

\begin{remark}
\noindent(1) If $n=2$, it is easy to check that \eqref{dddf} holds
with $$Q(t)=C(1+\|\nabla u_1\|^2+\|\Delta \phi_1\|^2+ \|\Delta
\theta_1\|^2+\|\nabla u_2\|^2+\|\Delta \phi_2\|^2).$$ On the other
hand, Corollary \ref{low-estimate} implies that $\int_t^{t+1} Q(s)
ds\leq C$ for all $t\geq 0$. As a consequence, we can actually
obtain the uniqueness of weak solutions to the problem
\eqref{navier-stokes}--\eqref{BC} for the two dimensional case.

(2) It is interesting to ask whether the problem
\eqref{navier-stokes}--\eqref{BC} has a weak-strong uniqueness
result in $3D$ as for the Navier--Stokes equations (c.f. Serrin
\cite{Se}). The main difficulty here is that due to the temperature-dependence
 of the surface tension coefficient $\lambda$, we lose
some dissipation in the derivation of the basic energy inequality
\eqref{basic energy inequality} in order to control the
corresponding higher-order nonlinear stress terms. When $\lambda$ is
a constant, one can obtain the weak-strong uniqueness result as in \cite{LL95} for a simplified liquid crystal system.
\end{remark}

\section{Long-time dynamics and stability}
\setcounter{equation}{0}

In this section, we shall discuss the long-time behavior of global
weak/strong solutions and stability properties of the problem
\eqref{navier-stokes}--\eqref{BC}. First, we present an alternative
result that indicates the eventual regularity of weak solutions in
$3D$ and is helpful to understand the long-time behavior of global
solutions.

\begin{proposition}
 \label{proposition on small data}
Suppose $n=3$. For any $(u_0, \phi_0, \theta_0)\in \mathbf{V}\times
 H^2(\Omega)  \times \big(H^1_0(\Omega) \cap H^2(\Omega)
\big)$ satisfying \eqref{ini} and
 \be
 \mathcal{A}_2(0)=\|\nabla u_0\|^2+a \lambda_0\|\Delta\phi_0-F'(\phi_0)\|^2+\|\Delta\theta_0\|^2 \leq R,\label{R}
 \ee
 where $R>0$ is a constant, there exists $\varepsilon_0>0$ depending on $\|u_0\|$, $\|\phi_0\|_{H^1}$,
$\|\theta_0\|_{H^1}$, $\Omega$, $R$, and coefficients of the system
such that either (i) the problem \eqref{navier-stokes}--\eqref{BC}
admits a unique global strong solution that is uniformly bounded in
time in $\mathbf{V}\times  H^2 \times (H_0^1\cap H^2)$, or (ii)
there is a $T_*\in (0, +\infty)$ such that
 $\mathcal{E}(T_*)<\mathcal{E}(0)-\varepsilon_0$.
\end{proposition}
\begin{proof} The proof follows from the idea in \cite{LL95}. For the convenience of the readers, we sketch it here.
Recalling the higher-order differential inequality \eqref{higher
energy inequality1}, we consider the following initial value problem
of an ordinary differential equation:
 \be
 \frac{d}{dt}Y(t)=C_*(Y^4(t)+Y(t)),\quad
Y(0)= R\geq \mathcal{A}_2(0).\label{ODE}
 \ee
We denote by $I=[0,T_{max})$ the maximal existence interval of
$Y(t)$ such that $ \displaystyle\lim_{t\rightarrow T_{max}^-}
Y(t)=+\infty.$ It is easy to check that $$0\leq \mathcal{A}_2(t)\leq
Y(t), \ \ \ \forall \, t \in I,$$
 which indicates $\mathcal{A}(t)$ exists on $I$. We note that
 $T_{max}$ is determined by $Y(0)=R$ and $C_*$ such that $T_{max}=T_{max}(R,C_*)$ is
 increasing when $R$ is decreasing. Taking $t_0=\frac34 T_{max}(R, C_*)> 0$, then
 we have
 \be
 0\leq \mathcal{A}_2(t)\leq Y(t)\leq K, \quad \forall\, t\in [0,
 t_0],\label{uniK}
 \ee
 where $K$ is a constant that only
depends on $R, C_*, t_0$. Take \be \varepsilon_0= \frac13
Rt_0\min\{\nu, \gamma, k\}. \label{epsilon} \ee If (ii) is not true,
namely, $ \mathcal{E}(t) \geq \mathcal{E}(0)-\varepsilon_0$ for all
$t\geq 0$, we infer from \eqref{basic energy inequality} that
 \[\int_0^{+\infty}\left(\frac{\nu}{2}\|\nabla u\|^2+a\lambda_0\gamma\|\Delta\phi-F'(\phi)\|^2+k\|\Delta\theta\|^2\right)dt
 \leq \varepsilon_0.
  \]
  Hence, there exists a
$t_* \in [\frac23 t_0,t_0]$ such that
 $$ \mathcal{A}_2(t_*) \leq  \max\Big\{\frac{2}{\nu}, \frac{1}{\gamma}, \frac{1}{k} \Big\} \frac{3\varepsilon_0}{t_0}= R.$$
 Taking $t_*$ as the initial time and
$Y(t_*)=R$ in \eqref{ODE}, then it follows from the above argument
that $Y(t)$ (and thus $\mathcal{A}_2(t)$) is uniformly bounded at
least on $[0,t_*+t_0]\supset [0, \frac53 t_0]$. Its bound remains
the same as that on $[0, t_0]$. By iteration, it follows that
$\mathcal{A}_2(t)$ is uniformly bounded for $t\geq 0$. Thus, we can
extend the (unique) local strong solution to infinity to get a
global one.
\end{proof}

\begin{proposition}[Eventual regularity of weak solutions in $3D$]\label{evereg}
When $n=3$, let $(u, \phi, \theta)$ be a global weak solution of the
problem \eqref{navier-stokes}--\eqref{BC}. Then there exists a time
$T_0 \in (0, +\infty)$ such that $(u, \phi, \theta)$ becomes a
strong solution in $[T_0, +\infty)$.
\end{proposition}
\begin{proof} We simply take $$R=1,\quad t_0=\frac34 T_{max}(R, C_*),\quad  \varepsilon_0= \frac13 t_0\min\left\{\frac{\nu}{2}, \gamma, k\right\}$$
in the proof of Proposition \ref{proposition on small data}. It
follows from \eqref{basic energy inequality} that there exist a $T_1>0$
such that
$$\int_{T_1}^{+\infty}\left(\frac{\nu}{2}\|\nabla u\|^2+a\lambda_0\gamma\|\Delta\phi-F'(\phi)\|^2+k\|\Delta\theta\|^2\right)dt
 \leq \varepsilon_0.$$ Then we can find a time $T_0\in [T_1, T_1+\frac13 t_0]$ such that
 $\mathcal{A}_2(T_0)\leq 1$ and $\mathcal{E}(t)-\mathcal{E}(T_0)\geq\mathcal{E}(t)-\mathcal{E}(T_1)\geq  -\varepsilon_0$ for all $t\geq T_0$. Taking $T_0$ as the initial time, we can apply Proposition \ref{proposition on small data}. The proof is complete.
\end{proof}


\subsection{Convergence to equilibrium}
We shall show the convergence of global solutions to single steady
states as time tends to infinity. Let $(u, \phi, \theta)$ be a
global weak solution of the problem
\eqref{navier-stokes}--\eqref{BC}. We infer from either Proposition
\ref{2dreg} ($n=2$) or Proposition \ref{evereg} ($n=3$) that after a
certain time $T>0$, the weak solution will be a strong one that is
uniformly bounded in $\mathbf{V}\times  H^2 \times H^2$ for all
$t\geq T$. Since we are now considering the long-time behavior as
$t\to +\infty$, we can simply use a shift in time and reduce our
study to the case of bounded strong solutions.

The main result of this subsection is as follows:

\begin{theorem} \label{theorem on long time behavior}
Suppose $n=2, 3$. Any bounded global strong solution $(u, \phi,
\theta)$ of the problem \eqref{navier-stokes}--\eqref{BC} converges
to a steady state $(\mathbf{0}, \phi_\infty, 0)$ as time goes to
infinity such that
 \be
\lim_{t\rightarrow +\infty}
 (\|u(t)\|_{\mathbf{H}^1}+\|\phi(t)-\phi_\infty\|_{H^2}+\|\theta(t)\|_{H^2})=0,\label{cgce}
 \ee
 where $\phi_\infty$ satisfies the following nonlinear
elliptic boundary value problem:
  \be  - \Delta \phi_\infty + F'(\phi_\infty)=0, \ \ \ x \in \Omega,\quad \text{with} \ \  \phi_\infty|_\Gamma=-1.
   \label{staa}
   \ee
 Moreover, we have the convergence rate
 \be
 \|u(t)\|_{\mathbf{H}^1}+\|\phi(t)-\phi_\infty\|_{H^2}+\|\theta(t)\|_{H^2}
  \leq C(1+t)^{-\frac{\xi}{(1-2\xi)}}, \quad \forall\ t \geq
 0.\label{rate}
 \ee
 $\xi \in (0,\frac12)$ is a constant depending on
 $\phi_\infty$. Furthermore, $\theta$ satisfies an exponential decay such
 that there exists a constant $C_0=C_0(n, \Omega)>0$,
 \be
\|\theta(t)\|\leq \|\theta_0\|e^{-C_0t}, \quad \forall \ t \geq
0.\label{decaytheta1}
 \ee
\end{theorem}
 \br\label{rrr}
 Decay properties of the velocity $u$ and temperature $\theta$ as time tends to infinity can be obtained by the energy method (see Proposition
\ref{proposition on decay property} below). However, convergence for the phase function $\phi$ is usually nontrivial
because the structure of the set of equilibria may be complicated
and the solutions to elliptic problem like \eqref{staa} may form a
continuum if the spatial dimension $n\geq 2$  (cf. e.g., Haraux \cite[Remark
2.3.13]{HA91}). As we have mentioned in the introduction, our results and their proofs hold for general Dirichlet boundary data for the phase function such that $-1$  in \eqref{BC2} can be replaced by a certain generic function $h(x)$. Since our problem enjoys a dissipative energy
inequality \eqref{basic energy inequality}, we can achieve the goal
by using the \L ojasiewicz--Simon approach (cf. e.g., \cite{J981, HJ99, GG10,
S83}). One advantage of this approach is that we can obtain the convergence result without investigating the structure of equilibria.
 \er

 The $\omega$-limit set of $(u_0, \phi_0, \theta_0)\in
\mathbf{V}\times H^2(\Omega) \times \big(H^1_0(\Omega) \cap
H^2(\Omega)\big)$ is defined as follows:
 \bea
 \omega(u_0, \phi_0, \theta_0) &= &\{ (u_\infty(x), \phi_\infty(x), \theta_\infty(x))\in \mathbf{V}\times
 H^2 \times (H^2\cap H_0^1):\non\\
 && \text{there
 \ exists\ } \{t_n\}\nearrow \infty  \text{\ such\ that\ } \non\\&&
 \ (u(t_n), \phi(t_n), \theta(t_n)) \rightarrow (u_\infty, \phi_\infty, \theta_\infty)\
 \text{in}\ \mathbf{L}^2 \times H^1 \times H^1\}.\non
 \eea
 \begin{proposition} \label{proposition on decay property}
Let $n=2,3$. For any global strong solutions to problem
\eqref{navier-stokes}--\eqref{BC}, there exists
$\mathcal{E}_\infty\geq 0$ such that \be
\lim_{t\to+\infty}\mathcal{E}(t)=\mathcal{E}_\infty, \label{conE}
\ee and it holds
 \be \lim_{t\rightarrow +\infty} (\|u(t)\|_{\mathbf{H}^1}+ \|\Delta \phi(t)-F'(\phi(t))\|+\|\Delta\theta(t)\|)=0.\label{conut}
 \ee
\end{proposition}
\begin{proof}
The total energy $\mathcal{E}(t)$ is nonnegative and decreasing as
$t$ increases (cf. \eqref{basic energy inequality}). Then
\eqref{conE} easily follows. For global bounded strong solution $(u,
\phi, \theta)$, we have $\mathcal{A}_1(t)\leq C$ ($n=2$) or
$\mathcal{A}_2(t)\leq C$ ($n=3$), then it follow from Lemma
\ref{high2d} ($n=2$) or Lemma \ref{high3da} ($n=3$) that $\frac{d
\mathcal{A}_i(t)}{dt} \leq C$ ($i=1,2$). On the other hand, we have
known from \eqref{basic energy inequality} that $\mathcal{A}_i(t)
\in L^1(0, +\infty)$. As a result, we can infer from Zheng
\cite[Lemma 6.2.1]{Z04} that $\lim_{t\to
+\infty}\mathcal{A}_i(t)=0$, which yields \eqref{conut}.
\end{proof}
 \begin{corollary} \label{proposition on omega lim set}
 $\omega(u_0, \phi_0, \theta_0)$ is a nonempty bounded subset in $\mathbf{V}\times
 H^2 \times (H^2\cap H_0^1)$.
 Moreover, $\omega(u_0, \phi_0, \theta_0)\subset \mathcal{S}=\big\{(\mathbf{0}, \tilde{\phi}, 0)\big| -\Delta\tilde{\phi}+F'(\tilde{\phi})=0 \ \mbox{in} \ \Omega,
\ \tilde{\phi}|_{\Gamma}=-1    \big\}$ and $\mathcal{E}=\mathcal{E}_\infty$ on $\omega(u_0, \phi_0, \theta_0)$.
 \end{corollary}

It is easy to verify that a critical point of the elastic energy
$E(\phi)$ given by \eqref{elastic} is equivalent to a solution to
the following elliptic boundary value problem \be  - \Delta \phi +
F'(\phi)=0,\quad x\in \Omega, \quad
   \phi|_\Gamma=-1.
   \label{staaq}
 \ee
We recall the following \L
 ojasiewicz--Simon type inequality (cf. Haraux-Jendoubi \cite{HJ99}).
 \bl [\L ojasiewicz--Simon inequality] \label{ls}
 Let $\psi$ be a critical point of $E(\phi)$. Then there exist constants
 $\xi\in(0,\frac12)$ and $\beta>0$ depending on $\psi$ such that
 for any $\phi\in H^1(\Omega)$, $\phi|_{\Gamma}=-1$ satisfying $\|\phi-\psi\|_{H^1(\Omega)}<\beta$, it
 holds
 \be
 \|-\Delta \phi+F'(\phi)\|_{H^{-1}}\geq
 |E(\phi)-E(\psi)|^{1-\xi}.\label{LSQ}
 \ee
 \el
 For any global bounded strong
solution, it follows from Corollary \ref{proposition on omega lim
set} that there is an increasing unbounded sequence
$\{t_n\}_{n\in\mathbb{N}}$ and a function $\phi_\infty\in
\mathcal{S}$ such that
   \be \lim_{t_n\rightarrow +\infty} \|\phi(t_n)-\phi_\infty\|_{H^1}
   =0. \label{secon}
   \ee
 As a result, we infer from Proposition \ref{proposition on decay property} and \eqref{secon} that
 \be \lim_{t_n\rightarrow +\infty} \mathcal{E}(t_n)=a\lambda_0E(\phi_\infty)=\mathcal{E}_\infty \
 \ \mbox{and} \ \ \mathcal{E}(t) \geq a\lambda_0 E(\phi_\infty), \ \forall \ t >
 0.
 \ee
If $\mathcal{E}(t_*)=a\lambda_0 E(\phi_\infty)$ for some $t_* > 0$,
then $\mathcal{E}(t)=\mathcal{E}_\infty$ for all $t\geq 0$. Thus, by
\eqref{basic energy inequality}, we have $
\|u(t)\|_{\mathbf{V}}=\|\theta(t)\|_{H^2}=\|\Delta\phi(t)-F'(\phi(t))\|
= 0$ for $t \geq t_0$. Besides, it follows from equation
\eqref{phase} that
 \be \|\phi_t\| \leq \|u\|_{\mathbf{L}^4}\|\nabla\phi\|_{\mathbf{L}^4}+\|\Delta\phi-F'(\phi)\| \leq C(\|\nabla u\|+\|\Delta\phi-F'(\phi)\|),\label{phit}
\ee thus $\|\phi_t(t)\|=0$ for $t \geq t_*$ and due to
\eqref{secon}, we have $\phi(t)=\phi_\infty$ for $t \geq t_0$.

Then we only have to consider the case
$\mathcal{E}(t)>\mathcal{E}_\infty= a\lambda_0 E(\phi_\infty)$, for
all $ t > 0$. Based on the sequential convergence \eqref{secon} and
the \L ojasiewicz--Simon inequality, by using the classical argument
in Jendoubi \cite{J981}, we can shown that after a certain time, the
trajectory $\phi(t)$ will fall into a certain small neighborhood of
$\phi_\infty$ and stay there for all time. Namely,

\begin{proposition} \label{proposition on trajectory}
There is a $t_0 > 0$, such that $\|\phi(t)-\phi_\infty\|_{H^1} <
\beta$, for all $t \geq t_0$.
\end{proposition}
\noindent Thus, for all $t \geq t_0$, $\phi(t)$ fulfills the
condition in Lemma \ref{ls}. Since $(1-\xi)
> \frac12$, then we infer from \eqref{LSQ} that
 \bea \big(\mathcal{E}(t)-a\lambda_0 E(\phi_\infty) \big)^{1-\xi} &\leq&
\Big(\|u\|^2+\zeta\|\nabla\theta\|^2+\omega\|\theta\|^2+a\lambda_0|E(\phi)-E(\phi_\infty)|
\Big)^{1-\xi} \non\\
&\leq&
C(\|u\|^2+\|\nabla\theta\|^2+\|\theta\|^2)^{1-\xi}+C|E(\phi)-E(\phi_\infty)|^{1-\xi}
\non\\
&\leq& C\|u\|+C\|\nabla\theta\|+C\|\Delta\phi-F'(\phi)\|, \non
 \eea which
combined with the energy inequality \eqref{basic energy
inequality} yields that for $t \geq t_0$, it holds
  \bea
-\frac{d}{dt}\big(\mathcal{E}(t)-\mathcal{E}_\infty \big)^{\xi}
&=&-\xi\big(\mathcal{E}(t)-a\lambda_0 E(\phi_\infty)
\big)^{\xi-1}\frac{d
\mathcal{E}}{dt} \non\\
&\geq& C\xi \frac{\|\nabla
u\|^2+\|\Delta\phi-F'(\phi)\|^2+\|\Delta\theta\|^2}{\|u\|+\|\nabla
\theta\|+\|\Delta\phi-F'(\phi)\|_{H^{-1}}}
\non\\
&\geq& C\big( \|\nabla u\|+\|\Delta\phi-F'(\phi)\|+\|\Delta\theta\|
\big). \label{application of LS}
 \eea
Integrating \eqref{application of LS} with respect to $t$, using
\eqref{phit} and the fact $\mathcal{E}(t) > \mathcal{E}_\infty$, we
have
  \be
 \int_{t_0}^\infty \|\phi_t(t)\| dt
 \leq C\int_{t_0}^\infty (\|\nabla u\|+\|\Delta\phi-F'(\phi)\|) dt \leq C(\mathcal{E}(t_0)-\mathcal{E}_\infty )^\xi<
 +\infty,\non
\ee which combined with the compactness of $\phi$ in $H^1$ yields
that $\lim_{t \rightarrow +\infty} \|\phi(t)-\phi_\infty\|_{H^1} =
0$.
Furthermore, since
 \bea
\|\Delta\phi-\Delta\phi_\infty\| &\leq&
\|\Delta\phi-\Delta\phi_\infty-F'(\phi)+F'(\phi_\infty)\|+\|F'(\phi)-F'(\phi_\infty)\|
\non\\
&\leq&\|\Delta\phi-F'(\phi)\|+\|\phi(t)-\phi_\infty\|_{H^1},
\label{kkk}
  \eea
we conclude  from \eqref{conut} that
 \be \displaystyle\lim_{t \rightarrow
+\infty} \|\phi(t)-\phi_\infty\|_{H^2} = 0.\non
 \ee
It remains to prove the convergence rate \eqref{rate}. By Lemma
\ref{ls} and \eqref{application of LS}, we obtain that
 \be
\frac{d}{dt}\big(\mathcal{E}(t)-\mathcal{E}_\infty\big)+
C\big(\mathcal{E}(t)-\mathcal{E}_\infty \big)^{2(1-\xi)}\leq 0,
\quad \forall\, t\geq t_0,   \non
 \ee which implies the decay rate for the total energy  $\mathcal{E}$
 \[ 0\leq \mathcal{E}(t)-\mathcal{E}_\infty\leq
 C(1+t)^{-\frac{1}{1-2\xi}},\quad \forall\, t\geq
 t_0. \]
Integrating \eqref{application of LS} on $(t,+\infty)$, where $t\geq
t_0$,  it follows from \eqref{phit} that
 \bea
 \int_t^{+\infty} \|\phi_t(\tau)\| d\tau
 &\leq&C(\mathcal{E}(t)-\mathcal{E}_\infty )^\xi \leq
 C(1+t)^{-\frac{\xi}{1-2\xi}}.\non
  \eea
Adjusting the constant $C$ properly, we get
 \be
    \|\phi(t)-\phi_\infty\|\leq C(1+t)^{-\frac{\xi}{1-2\xi}}, \quad \forall \, t\geq 0.\label{rate1}
 \ee
Higher-order estimates on the convergence rate can be achieved by
constructing proper differential inequalities via energy method. It
is clear that for the asymptotic limit $(\mathbf{0}, \phi_\infty,
0)$, the system \eqref{navier-stokes}--\eqref{temperature} is
reduced to \bea
 \nabla P_\infty+\frac{1}{2}\nabla\big(|\nabla \phi_\infty|^2\big)&=&-\nabla \phi_\infty\cdot \Delta \phi_\infty,\label{sta1}\\
 -\Delta \phi_\infty+F'(\phi_\infty)&=&0,\ \ \text{with} \ \phi_\infty|_{\Gamma}=-1.\label{sta2}
 \eea
Denote $\varphi=\phi-\phi_\infty$. Then $(u, \varphi, \theta)$
satisfies
 \bea
&&u_t+u\cdot\nabla u+\nabla
\tilde{p}-\nu\Delta{u} \non\\
&=&-a\lambda_0(\Delta\varphi\nabla\phi+
\Delta\phi_\infty\nabla\varphi)
+b\lambda_0\nabla\cdot(\theta\nabla\phi\otimes\nabla\phi)
+\alpha\theta g\mathbf{j},
\label{11}   \\
&&\nabla\cdot u=0, \label{11a} \\
&&\varphi_t+u\cdot\nabla\phi=\gamma\Delta\varphi-\gamma\big(F'(\phi)-F'(\phi_\infty)\big),
\label{22}   \\
&&\theta_t+u\cdot\nabla\theta=k\Delta\theta, \label{33}
 \eea
where we absorb all those gradient terms into the modified pressure
$\tilde{p}$.

 Multiplying \eqref{11} by $u$, \eqref{22} by
$a\lambda_0(-\Delta\varphi+\big(F'(\phi)-F'(\phi_\infty)\big))+\varphi$
and \eqref{33} by $-\Delta\theta$ respectively, integrating over
$\Omega$ and adding them together, we have
 \bea
 &&\frac12\frac{d}{dt} \mathcal{Y}(t)+\nu\|\nabla{u}\|^2+a\lambda_0\gamma\|\Delta \phi-F'(\phi)\|^2+\gamma\|\nabla \varphi\|^2
 +k\|\Delta\theta\|^2 \non\\
 &=&
 -b\lambda_0\int_{\Omega}\theta\nabla_i\phi\nabla_j\phi\nabla_ju_i\,dx+\alpha g(\theta \mathbf{j},
 u)+ (u\cdot \nabla \theta, \Delta \theta)\non\\
 &&  -(u\cdot\nabla\phi,
\varphi)-\gamma\big(F'(\phi)-F'(\phi_\infty),
\varphi \big) \non\\
 &:=&\sum_{m=1}^5 R_m,
 \label{ra1}
 \eea
 where
 \be
   \mathcal{Y}(t)=\|u\|^2+a\lambda_0\|\nabla \varphi\|^2+2a\lambda_0\int_\Omega [F(\phi)-F(\phi_\infty)
   - F'(\phi_\infty)\varphi]dx+\|\varphi\|^2+ \|\nabla \theta\|^2.
   \label{def of Y(t)}
   \ee
 In the derivation of \eqref{ra1}, we have used \eqref{sta1}, \eqref{sta2} and the following fact
 \bea
 && \int_\Omega(\Delta\varphi\nabla\phi+ \Delta\phi_\infty\nabla\varphi)\cdot u dx+\int_\Omega u\cdot \nabla \phi \left[-\Delta \varphi+\big(F'(\phi)-F'(\phi_\infty)\big)\right]dx\non\\
 &=& \int_\Omega (\Delta\phi_\infty-F'(\phi_\infty))\nabla\phi\cdot u dx +\int_\Omega u\cdot \nabla F(\phi) dx-\int_\Omega \Delta\phi_\infty\nabla\phi_\infty\cdot u dx\non\\
 &=&0.\non
 \eea
Since we are now dealing with global strong solutions that are
uniformly bounded in $\mathbf{V}\times H^2\times H^2$, it follows
that
 \bea R_1 &\leq&
\|\theta\|_{L^\infty}\|\nabla u\|\|\nabla\phi\|_{\mathbf{L}^4}^2
\leq
C\|\Delta\theta\|^\frac34\| \theta\|^\frac14\|\nabla u\| \|\phi\|_{H^2}^2\non\\
&\leq& \frac{\nu}{12}\|\nabla
u\|^2+ \frac{k}{4}\|\Delta\theta\|^2+ C\|\theta\|^2,  \non\\
R_2 &\leq& |\alpha| |g|\|\theta\|\|u\|\leq \frac{\nu}{12}\|\nabla
u\|^2+C\|\theta\|^2,  \non\\
R_3 &\leq&
\frac{k}{8}\|\Delta\theta\|^2+C\|u\|_{\mathbf{L}^6}^2\|\nabla\theta\|_{\mathbf{L}^3}^2
\leq \frac{k}{8}\|\Delta\theta\|^2+C\|\theta\|^{\frac12}\|\Delta\theta\|^{\frac32} \non\\
&\leq& \frac{k}{4}\|\Delta\theta\|^2+C\|\theta\|^2, \non
 \\
 R_4+R_5&\leq&
\|u\|_{\mathbf{L}^6}\|\nabla\phi\|_{\mathbf{L}^3}\|\varphi\|+C\|\varphi\|^2
 \leq  \frac{\nu}{12}\|\nabla
u\|^2+C\|\varphi\|^2.\non
 \eea
From the definition of $F(\phi)$, we have
 $\left\vert\int_\Omega
 [F(\phi)-F(\phi_\infty)-F'(\phi_\infty)\varphi]
         dx\right\vert\leq
          C_1\|\varphi\|^2$.
Combined with the definition of $\mathcal{Y}(t)$ in \eqref{def of Y(t)}, it
yields that
 \be \mathcal{Y}(t)+C_1\|\varphi\|^2 \geq
C(\|u\|^2+\|\varphi\|_{H^1}^2+\|\nabla \theta\|^2). \label{ra6}
 \ee
  It follows from \eqref{ra1},
\eqref{ra6} and the estimates on $R_m$ ($m=1,...,5$) that
 \be
\frac{d}{dt}\mathcal{Y}(t)+C_2\mathcal{Y}(t)+C_3\mathcal{A}_i(t)
\leq C(\|\varphi\|^2+\|\theta\|^2), \quad i=1,2. \label{ra4a}
 \ee
Recalling Lemma \ref{high2d} ($n=2$) or Lemma \ref{high3da} ($n=3$),
we have
 \be
 \frac{d}{dt}\mathcal{A}_i(t) \leq C_4\mathcal{A}_i(t). \label{simplified high energy inequality}
 \ee
Multiplying \eqref{simplified high energy inequality} with
$\eta=\frac{C_3}{2C_4}$, and adding the resultant to \eqref{ra4a},
we get from \eqref{decaytheta1} and \eqref{rate1} that
 \be
\frac{d}{dt}\big[\mathcal{Y}(t)+\eta\mathcal{A}_i(t) \big]+
C'\big[\mathcal{Y}(t)+\eta\mathcal{A}_i(t) \big] \leq
C(\|\varphi\|^2+\|\theta\|^2),\quad
 \forall\, t \geq 0.\non
 \ee
Consequently,
 \bea &&\mathcal{Y}(t)+\eta\mathcal{A}_i(t) \non\\
&\leq&Ce^{-C't}+Ce^{-C't}\Big(\int_0^{\frac{t}{2}}e^{C'\tau}(1+\tau)^{-\frac{2\theta}{1-2\theta}}d\tau
+\int_{\frac{t}{2}}^{t}e^{C'\tau}(1+\tau)^{-\frac{2\theta}{1-2\theta}}d\tau  \Big)   \non\\
&=&Ce^{-C't}+Ce^{-C't}\int_0^{\frac{t}{2}}e^{C'\tau}(1+\tau)^{-\frac{2\theta}{1-2\theta}}d\tau
 \non\\
&&+Ce^{-C't}\left[\frac{e^{C'\tau}}{C'}(1+\tau)^{-\frac{2\theta}{1-2\theta}}\Big|^{\tau=t}_{\tau=\frac{t}{2}}
+\frac{2\theta}{C'(1-2\theta)}\int_0^te^{C'\tau}(1+\tau)^{-\frac{1}{1-2\theta}}d\tau \right]   \non\\
&\leq&Ce^{-C't}+Ce^{-C't}\Big(e^{\frac{C't}{2}}\int_0^{\frac{t}{2}}(1+\tau)^{-\frac{2\theta}{1-2\theta}}d\tau
+(1+t)^{-\frac{2\theta}{1-2\theta}}e^{C't}
\Big) \non\\
 &\leq&C(1+t)^{-\frac{2\xi}{1-2\xi}}, \ \ \ \ \ \ \ \ \ \forall\, t
\geq 0.\label{conAB}
 \eea
Then our conclusion \eqref{rate} follows from  \eqref{conAB}, the
definitions of $\mathcal{A}_i(t)$, $\mathcal{Y}(t)$ and \eqref{kkk}.
The exponential decay of $\theta$ (see \eqref{decaytheta1}) easily follows
from \eqref{part2 of basic energy law}. The proof of Theorem
\ref{theorem on long time behavior} is complete.


\subsection{Stability of energy minimizers}
We have shown that any global weak (or strong) solution of problem
\eqref{navier-stokes}--\eqref{BC} will converge to a steady state as
time goes to infinity (without smallness restrictions on the initial data).
However, it is not clear to which equilibrium it will converge,
since the set of equilibria may be a continuum (for general Dirichlet boundary data of the phase function). This is different
from the classical concept of stability in the literature. Below we
shall show that if $u_0$ and $\theta_0$ are close to zero and
$\phi_0$ is near a certain local minimizer of the elastic energy
$E(\phi)$, problem \eqref{navier-stokes}--\eqref{BC} admits a unique
global strong solution. Moreover, the energy minimizer is Lyapunov
stable. The same as in Remark \ref{rrr}, the results and their proofs in this subsection actually hold for general Dirichlet boundary conditions for $\phi$, not only the special case \eqref{BC2}.

 \bd \label{definition of local minimizer} The function $\phi^\ast \in \mathcal{K}:=\{\phi\in H^1(\Omega):\ \phi|_\Gamma=-1\}$ is
called a local minimizer of $E(\phi)$, if there exists $\sigma > 0$,
such that for any $\phi \in \mathcal{K}$ satisfying
$\|\phi-\phi^\ast\|_{H^1} \leq \sigma$, it holds $E(\phi) \geq
E(\phi^\ast)$.
 \ed

\begin{remark}
It is easy to verify that any local minimizer of $E(\phi)$ is a
critical point of $E(\phi)$ and satisfies the elliptic boundary
value problem \eqref{staaq}.
\end{remark}

\begin{theorem} \label{theorem on small data}
Suppose $n=3$ and \eqref{ini} is satisfied. Let $\phi^\ast \in
H^2(\Omega)\cap \mathcal{K}$ be a local minimizer of $E(\phi)$.
For arbitrary $r>0$, we consider the set
 \bea
\mathcal{B}_r&=&\big\{(u, \phi, \theta)\in \mathbf{V} \times
(H^2(\Omega)\cap \mathcal{K})
\times \big( H^2(\Omega)\cap H^1_0(\Omega) \big) \big\}:\non\\
&& \quad  \|u\|_{\mathbf{H}^1} \leq r, \ \|\phi-\phi^\ast\|_{H^2}
\leq r,\ \|\theta\|_{H^2} \leq r\}.
 \eea Then there exist positive
constants $\sigma_1, \sigma_2, \sigma_3$ depending on $r$, $\Omega$,
$\sigma$, $\nu^*$,  $\phi^\ast$ and coefficients of the system, such
that for any initial data $(u_0, \phi_0, \theta_0) \in
\mathcal{B}_r$ satisfying
 \be  \|u_0\| \leq \sigma_1,\ \|\phi_0-\phi^\ast\|_{H^1} \leq \sigma_2,\ \|\theta_0\|_{H^1} \leq
\sigma_3, \ee we have

(i) the problem \eqref{navier-stokes}--\eqref{BC} admits a unique
global strong solution $(u, \phi, \theta)$;

(ii) the energy minimizer $\phi^*$ is Lyapunov stable;

(iii) the global strong solution has the same long-time behavior as
in Theorem \ref{theorem on long time behavior}. Although the limit
function $\phi_\infty$ may differ from the minimizer $\phi^*$,  the
total energy $\mathcal{E}(t)$ will converge to the same energy level
of $\phi^\ast$
 \be
  \lim_{t\to+\infty}\mathcal{E}(t)=a\lambda_0E(\phi_\infty)=a\lambda_0 E(\phi^\ast).\label{leee}
 \ee
 Moreover, if $\phi^*$ is an \textit{isolated} local minimizer, then $\phi_\infty=\phi^*$
 and thus $\phi^*$ is asymptotically stable.
\end{theorem}

\begin{proof}
By Proposition \ref{proposition on small data}, in order to prove
the existence of global strong solutions, we only have to verify
that
 \be
\mathcal{E}(t)-\mathcal{E}(0) \geq -\varepsilon_0, \ \ \forall \ t
\in [0, +\infty),   \label{tedro}
 \ee
  where $\varepsilon_0$ is
defined as in \eqref{epsilon}. We know from the argument in
Proposition \ref{proposition on small data} that there exists
$t_0=\frac34T_{max}$ and $\mathcal{A}_2(t)$ is uniformly bounded on
$[0, t_0]$ by a constant depending on $r$, $\phi^*$, $\Omega$ and
coefficients of the system. Since $\mathcal{A}_2(t)$ is bounded on
$[0, t_0]$, it holds
 \bea
&& \mathcal{E}(t)-\mathcal{E}(0)\non\\
&=&\frac12\|u(t)\|^2-\frac12\|u_0\|^2+a\lambda_0(E(\phi(t))-E(\phi_0))
+\frac{\zeta}{2}\|\nabla\theta(t)\|^2+\frac{\omega}{2}\|\theta(t)\|^2\non\\
&&
-\frac{\zeta}{2}\|\nabla\theta_0\|^2-\frac{\omega}{2}\|\theta_0\|^2
\non\\
&\geq&-\frac12\|u_0\|^2+a\lambda_0(E(\phi(t))-E(\phi^\ast)+E(\phi^\ast)-E(\phi_0))
-\frac{\zeta}{2}\|\nabla\theta_0\|^2-\frac{\omega}{2}\|\theta_0\|^2 \non\\
&\geq&-\frac12\|u_0\|^2
-\frac{\zeta}{2}\|\nabla\theta_0\|^2-\frac{\omega}{2}\|\theta_0\|^2
-C_1\|\phi_0-\phi^\ast\|_{H^1}+a\lambda_0(E(\phi(t))-E(\phi^\ast)),\non
 \eea
 where $\zeta, \omega$ are as in Proposition \ref{BEL} and $C_1$ depends on $r$, $\phi^*$, $\Omega$, and coefficients of the system.
    We assume that $\sigma_m$ ($m=1, 2, 3$) are sufficiently small such that
\be \frac12\sigma_1^2+ \frac12\max\{\zeta,
\omega\}\sigma_3^2+C_1\sigma_2\leq \varepsilon_0.\label{sig1} \ee
If we can ensure that
 \be
  E(\phi(t)) - E(\phi^\ast)\geq 0, \ \ \forall \ t \in [0, t_0],  \label{nearc}
 \ee
 then we  have
  \be \mathcal{E}(t)-\mathcal{E}(0)
\geq -\varepsilon_0, \ \ \forall \ t \in [0, t_0]. \label{tedro1}
\ee This enables us to apply the argument in the proof for
Proposition \ref{proposition on small data} to extend the local
strong solution from $[0, t_0]$ to $[0, t_0+\frac{2t_0}{3}]=[0,
\frac{5t_0}{3}]$.

 By Definition
\ref{definition of local minimizer}, \eqref{nearc} can be reduced to
the following condition
 \be
 \|\phi(t)-\phi^\ast\|_{H^1}< \mbox{min}\{\sigma, \beta\}:=\delta, \ \ \forall \ t \in[0,
 t_0],\label{nlm}
 \ee
 where $\beta>0$ is the constant depending on $\phi^*$ given in Lemma \ref{ls} (taking $\psi=\phi^*$ therein, we note that $\phi^*$ is a critical point of $E(\phi)$). We shall show that one can choose a smaller $\sigma_2$ satisfying
\be \sigma_2\leq \frac14\delta,\label{sig2} \ee such that
\eqref{nlm} holds. This can be done via the \L ojasiewicz--Simon
approach  by a contradiction argument (cf. Wu et al \cite{LWX11}).
If \eqref{nlm} is not true, then by the continuity of $\phi$ that
$\phi\in C([0,t_0]; H^1_0)$, there exists a minimal time $T_0\in (0,
t_0]$, such that $\|\phi(T_0)-\phi^\ast\|_{H^1}=\delta$. We observe
that $\mathcal{E}(t) \geq a\lambda_0 E(\phi^\ast)$ for any $t \in
[0, T_0]$.  If for some $T\leq T_0$,
$\mathcal{E}(T)=a\lambda_0E(\phi^\ast)$, then we deduce from the
definition of the local minimizer and the basic energy inequality
\eqref{basic energy inequality} that for $t\geq T$, $\mathcal{E}(t)$
cannot drop and will remain $a\lambda_0E(\phi^\ast)$. Thus, $\nabla
u=\Delta\phi-F'(\phi)=\Delta\theta \equiv 0$ for all $t \geq T$ and
the evolution becomes stationary. The conclusion easily follows. In
the following, we just assume $\mathcal{E}(t) >
a\lambda_0E(\phi^\ast)$ for $t \in [0, T_0]$. Applying Lemma
\ref{ls} with $\psi=\phi^\ast$, we get (similar to
\eqref{application of LS})
 \be   -\frac{d}{dt}[\mathcal{E}(t)-a\lambda_0E(\phi^\ast)]^\xi \geq C(\|\nabla u\|+\|\Delta \phi-F'(\phi)\|+\|\Delta \theta\|), \quad \forall \, t\in (0,
T_0).\non
 \ee
Then we infer from \eqref{phit} that
 \bea
 \|\phi(T_0)-\phi_0\|_{H^1} &\leq&
 C\|\phi(T_0)-\phi_0\|^{\frac12}\|\phi(T_0)-\phi_0\|_{H^2}^{\frac12}\non\\
 &\leq& C\Big( \int_{0}^{T_0}\|\phi_t(t)\|dt \Big)^\frac12 \leq
C[\mathcal{E}(0)-a\lambda_0E(\phi^\ast)]^{\frac{\xi}{2}}\non\\
&  \leq & C_2\Big(
\|u_0\|^2+\|\theta_0\|^2_{H^1}+\|\phi_0-\phi^\ast\|_{\mathbf{H}^1}
\Big)^\frac{\xi}{2}.
 \eea
Choosing $\sigma_m$ ($m=1, 2, 3$) satisfying \eqref{sig1},
\eqref{sig2} and
 \be
 C_2\Big(
\sigma_1^2+\sigma_3^2+\sigma_2 \Big)^\frac{\xi}{2}\leq
\frac12\delta, \label{sig3}
 \ee
 we can see that
 \be \|\phi(T_0)-\phi^\ast\|_{H^1} \leq
\|\phi(T_0)-\phi_0\|_{H^1}+\|\phi_0-\phi^\ast\|_{H^1} \leq \frac34
\delta <\delta, \non \ee
which leads to a contradiction with the definition of $T_0$. Thus,
\eqref{nlm} is true and  \eqref{tedro1} holds.

By iteration, we conclude that the local strong solution $(u, \phi,
\theta)$ can be extended by a fixed length $\frac{2t_0}{3}$ in each
step and it is indeed a global solution with $\mathcal{A}_2(t)$
being uniformly bounded. Then by Theorem \ref{theorem on long time
behavior}, there exists a critical point $\phi_\infty$ of $E(\phi)$,
such that the global solution $(u(t), \phi(t), \theta(t))$ satisfies
the same long-time behavior \eqref{cgce} with convergence rate
\eqref{rate}. It is easy to see from the above argument that for any
$\epsilon>0$, by choosing sufficiently small $\sigma_m$ ($m=1, 2,
3$), it holds $\|\phi(t)-\phi^\ast\|_{H^1} \leq \epsilon $, for all
$t\geq 0$. This implies the Lyapunov stability of the local
minimizer $\phi^*$. In particular, we have
 \be \|\phi_\infty-\phi^\ast\|_{H^1}
\leq \|\phi(t)-\phi_\infty\|_{H^1}+\|\phi(t)-\phi^\ast\|_{H^1} \leq
\mbox{min}\{\sigma, \beta \}.\non
 \ee
 Applying the \L
ojasiewicz--Simon inequality once more with $\psi=\phi^\ast$, we
conclude that \be |E(\phi_\infty)-E(\phi^\ast)|^{1-\xi}\leq
 \|-\Delta \phi_\infty+F'(\phi_\infty)\|=0,\label{EsED}
 \ee
which together with \eqref{cgce} yields \eqref{leee}. The proof is
complete.
\end{proof}

\bigskip
\noindent \textbf{Acknowledgments:} The authors wish to thank the referees for
their very helpful comments and suggestions on an earlier version of this paper. Part of the work was done when
Xu was visiting School of Mathematical Sciences at Fudan
University, whose hospitality is acknowledged. Wu was partially
supported by NSF of China 11001058, SRFDP, the Fundamental Research
Funds for the Central Universities. Xu was partially supported by
NSF grant DMS-0806703.

\end{document}